\documentclass[11pt,a4paper]{amsart}
\usepackage[utf8]{inputenc}
\usepackage[english]{babel}
\usepackage{amsmath}
\usepackage{amsfonts}
\usepackage{amssymb}
\usepackage{mathabx}
\usepackage{graphicx}
\usepackage{pstricks}
\usepackage{float}
\usepackage{fullpage}
\usepackage{hyperref}

\newcommand{\R}{\mathbb R}
\newcommand{\C}{\mathbb C}
\newcommand{\E}{\mathbb E}

\newcommand{\N}{\mathbb N}

\newcommand{\1}{\mathbf 1}
\newcommand{\U}{\mathcal U}
\newcommand{\Ub}{\mathbb U}
\newcommand{\Tr}{\operatorname{Tr}}
\renewcommand{\Re}{\operatorname{Re}}
\newcommand{\I}{\mathfrak I}

\newtheorem{thm}{Theorem}[section]
\newtheorem{lemma}[thm]{Lemma}

\newtheorem{cor}[thm]{Corollary}

\newtheorem{propB}{Proposition}

\theoremstyle{remark}
\newtheorem{rem}[thm]{Remark}

\numberwithin{equation}{section}

\author{Titus Lupu}
\address {CNRS and LPSM, UMR 8001,
Sorbonne Université,
4 place Jussieu,
75252 Paris cedex 05,
France}
\email
{titus.lupu@upmc.fr}

\title{Topological expansion in isomorphism
theorems between matrix-valued fields and random walks}

\begin{document}

\begin{abstract}
We consider Gaussian fields of real symmetric, complex Hermitian or quaternionic Hermitian matrices over an electrical network, 
and describe how the isomorphisms between these fields and random walks give rise to topological expansions encoded by ribbon graphs. 
We further consider matrix-valued Gaussian fields twisted by an orthogonal, unitary or symplectic connection. In this case the isomorphisms involve traces of holonomies of the connection along random walk loops parametrized by boundary cycles of ribbon graphs.
\end{abstract}


\maketitle

\section{Introduction}

It is known since the works of Symanzik
\cite{Symanzik65Scalar,Symanzik66Scalar,Symanzik1969QFT}
and Brydges, Fröhlich and Spencer \cite{BFS82Loop} that the Gaussian free field (GFF) has representations involving random walks, sometimes referred to as "isomorphisms". 
For a survey on the subject we refer to
\cite{MarcusRosen2006MarkovGaussianLocTime,
Sznitman2012LectureIso}. Here we will be interested in a representation that appears in 
Brydges, Fröhlich, Spencer \cite{BFS82Loop} and 
Dynkin \cite{Dynkin1984Isomorphism,Dynkin1984IsomorphismPresentation}, that expresses
\begin{displaymath}
\mathbb{E}\Big[
\prod_{k=1}^{2r}\phi(x_{k})
F(\phi^{2}/2)
\Big],
\end{displaymath}
for $\phi$ a GFF, in terms of pairings of vertices
$x_{k}$-s and random walks joining the pairs. Following
\cite{BauerschmidtHelmuthSwan2}, we will call it the
BFS-Dynkin isomorphism.

Kassel and Lévy considered vector-valued GFFs twisted by an orthogonal or unitary connection \cite{KasselLevy16CovSym}. 
In this setting isomorphism theorems involve the holonomy of the connection along the random walks. 
Holonomies along random walk or Brownian loops have been also studied in
\cite{LeJanHolonomiesRW,LeJanTopologyBrownianLoops,
CamiaLeJanReddy2020}.

In this paper we will consider fields of random Gaussian matrices, real symmetric, complex Hermitian
or quaternionic Hermitian,
on an electrical network. These are matrix-valued GFFs.
The matrix above any vertex of the network is proportional
to a GOE, GUE or GSE matrix. 
Here we will write an isomorphism for
\begin{equation}
\label{EqI2}
\Big\langle \Big(\prod_{l=1}^{m(\nu)}
\Tr\Big(
\prod_{k=\nu_{1}+\dots+\nu_{l-1}+1}
^{\nu_{1}+\dots+\nu_{l}}
\Phi(x_{k})\Big)\Big)
F(\Tr(\Phi^{2})/2)\Big\rangle_{\beta,n},
\end{equation}
where $\Phi$ is the matrix-valued GFF,
$\beta\in\{1,2,4\}$, $n$ is the size of the matrices,
$\nu_{1},\dots, \nu_{m(\nu)}$
are positive integers with
$\vert\nu\vert:=\nu_{1}+\dots +\nu_{m(\nu)}$ even,
and $x_{1},\dots,x_{\vert\nu\vert}$ vertices on the network.
By taking the $x_{k}$-s equal inside each of the traces, we get  symmetric polynomials in the eigenvalues. By expanding the traces and the product above, one can write a BFS-Dynkin's isomorphism for each of the terms of the sum. However,
one gets many different terms that give identical contributions, many terms with contributions that cancel out, being of opposite sign, and many terms that give zero contribution. By regrouping the terms surviving to cancellation into powers of $n$, 
one gets a combinatorial structure known as a topological expansion.
The terms of the expansion correspond to ribbon graphs 
with $m(\nu)$ vertices,  obtained by pairing and gluing 
$\vert\nu\vert$ ribbon half-edges. Each gluing may be straight or twisted. The power of $n$ is then given by the number of 
cycles formed by the boundary components of the ribbons. 
It can be also expressed in terms of genera of compact surfaces, orientable or not.

The topological expansion has been introduced by 't Hooft for the study of Quantum Chromodynamics \cite{tHooft74TopologicalExpansion},
and further developed by Brézin, Itzykson, Parisi and Zuber
\cite{BIPZ78PlanarDiagrams,IZ80PlanarApproximation}.
Nowadays there is a broad, primarily physics literature on this topic. In particular, topological expansion of one matrix or several matrix integrals is used for the enumeration of maps on surfaces and other graphical objects
\cite{BIZ80GraphicalEnumeration,Zvonkin97MatrixIntegrals,
LandoZvonkine04GraphsSurfaces,Eynard16CountingSurfaces}.
Compared to the case of one matrix integrals, where each ribbon edge comes only with a scalar weight, in our setting each ribbon edge will be associated to a measure on random walk paths between two vertices $x_{k}$ and $x_{k'}$ on the network.
For an introduction to the topological expansion we refer to\cite{Zvonkin97MatrixIntegrals,
EynardKimuraRibault18RandomMatrices}.

We will further extend our framework and consider matrix-valued free fields twisted by a connection of orthogonal ($\beta=1$),
unitary ($\beta=2$) or symplectic ($\beta=4$) matrices. We rely for this on results of Kassel and Lévy for twisted vector-valued GFFs \cite{KasselLevy16CovSym}. If $\Phi$ is the matrix-valued GFF twisted by a connection $U$ and 
$(\lambda_{1},\dots,\lambda_{n})$
its fields of eigenvalues, then the isomorphism for
\begin{displaymath}
\Big\langle \prod_{l=1}^{m(\nu)}
\Big(\sum_{i=1}^{n}
\lambda_{i}(x_{\nu_{1}+\dots +\nu_{l}})^{\nu_{l}}
\Big)
F\Big(\dfrac{1}{2}\sum_{i=1}^{n}
\lambda_{i}^{2}\Big)\Big\rangle_{\beta,n}^{U}
\end{displaymath}
involves a topological expansion where instead of $n$ to the power the number of cycles in a ribbon graph appears a product of traces of holonomies of the connection, one per each boundary cycle in the ribbon graph. The holonomies are taken along loops made of concatenated random walk paths. Such traces of holonomies along loops are called Wilson loop observables.

\section{Preliminaries}
\label{SecPrelim}

\subsection{Quaternions, symplectic matrices and quaternionic Hermitian matrices}

$\mathbb{H}$ will denote the skew (non-commutative) field of quaternions. Its elements are of the form
\begin{displaymath}
q=q_{r} + q_{i}\textbf{i} + q_{j} \textbf{j} + q_{k} \textbf{k},
\end{displaymath}
where $q_{r},q_{i},q_{j},q_{k}\in\mathbb{R}$, and 
$\textbf{i}$, $\textbf{j}$ and $\textbf{k}$ satisfy the relations
\begin{displaymath}
\textbf{i}^{2}=\textbf{j}^{2}=\textbf{k}^{2}=-1,
\end{displaymath}
\begin{displaymath}
\textbf{i}\textbf{j}=-\textbf{j}\textbf{i}=\textbf{k},
\qquad
\textbf{j}\textbf{k}=-\textbf{k}\textbf{j}=\textbf{i},
\qquad
\textbf{k}\textbf{i}=-\textbf{i}\textbf{k}=\textbf{j}.
\end{displaymath}
The coefficient $q_{r}$ is the real part $\Re(q)$ of the quaternion $q$.
The algebra of quaternions has the usual representation over $2\times 2$ complex matrices:
\begin{displaymath}
\mathtt{C}(q)=
\left(
\begin{array}{cc}
q_{r}+q_{i}\textbf{i} & -q_{j}-q_{k}\textbf{i} \\ 
q_{j}-q_{k}\textbf{i} & q_{r}-q_{i}\textbf{i}
\end{array} 
\right).
\end{displaymath}
The conjugate of a quaternion is given by
\begin{displaymath}
\bar{q}=q_{r} - q_{i}\textbf{i} - q_{j} \textbf{j} - q_{k} \textbf{k}.
\end{displaymath}
The matrix $\mathtt{C}(\bar{q})$ is the adjoint of  
$\mathtt{C}(q)$ for the Hermitian inner product:
\begin{displaymath}
\mathtt{C}(\bar{q})=\mathtt{C}(q)^{\ast}.
\end{displaymath}
If $q_{1},q_{2}\in\mathbb{H}$,
\begin{displaymath}
\overline{q_{1}q_{2}}=
\bar{q_{2}}\bar{q_{1}}.
\end{displaymath}
The absolute value $\vert q\vert$ is given by
\begin{displaymath}
\vert q\vert^{2}=
q_{r}^{2}+q_{i}^{2}+q_{j}^{2}+q_{k}^{2}=
q\bar{q}=\bar{q}q=\det(\mathtt{C}(q)).
\end{displaymath}
For more on quaternions, we refer to \cite{HandbookQuaternions14}.

We will denote by $\mathcal{M}_{n}(\mathbb{H})$ the ring of $n\times n$ matrices with quaternionic entries. The product of matrices is defined in the same way as for matrices over a commutative field:
\begin{displaymath}
(AB)_{ij}=\sum_{k=1}^{n}A_{ik}B_{kj},\qquad
i,j\in\{1,\dots,n\}.
\end{displaymath}
One can associate to a matrix $M\in\mathcal{M}_{n}(\mathbb{H})$ a
$2n\times 2n$ matrix with complex entries, by replacing each entry
$M_{ij}$ by a $2\times 2$ block
$\mathtt{C}(M_{ij})$. The resulting matrix in
$\mathcal{M}_{2n}(\mathbb{C})$ will be again denoted
$\mathtt{C}(M)$. 
The map $\mathtt{C}$ is then a morphism of rings from
$\mathcal{M}_{n}(\mathbb{H})$ to
$\mathcal{M}_{2n}(\mathbb{C})$. 

The trace of a matrix of quaternions is defined as usually:
\begin{displaymath}
\Tr(M)=\sum_{i=1}^{n}M_{ii}.
\end{displaymath}
However, it is more convenient to deal with the real part of the trace, as
\begin{displaymath}
\Re(\Tr(M))=\dfrac{1}{2}\Tr(\mathtt{C}(M)),
\end{displaymath}
and for $A,B\in \mathcal{M}_{n}(\mathbb{H})$,
\begin{displaymath}
\Re(\Tr(AB))=\Re(\Tr(BA)).
\end{displaymath}
Note that in general there is no equality between
$\Tr(AB)$ and $\Tr(BA)$, since the product of quaternions is not commutative.

The quaternion adjoint of a matrix $M\in \mathcal{M}_{n}(\mathbb{H})$, denoted $M^{\ast}$, extends the notion of adjoint for matrices with complex entries:
\begin{displaymath}
(M^{\ast})_{ij}=\overline{M_{ji}},
\qquad i,j\in\{1,\dots,n\},
\end{displaymath}
where $\overline{M_{ji}}$ is the quaternion conjugate. We have that
\begin{displaymath}
\mathtt{C}(M^{\ast})=\mathtt{C}(M)^{\ast},
\end{displaymath}
where on the left-hand side $^{\ast}$ denotes the quaternion adjoint, and on the right-hand side, $^{\ast}$ denotes the complex adjoint. If 
$A,B\in \mathcal{M}_{n}(\mathbb{H})$,
\begin{displaymath}
(AB)^{\ast}=B^{\ast}A^{\ast}.
\end{displaymath}

The quaternionic unitary group, $U(n,\mathbb{H})$ is the set of matrices
$U\in \mathcal{M}_{n}(\mathbb{H})$ satisfying
\begin{displaymath}
U U^{\ast} = I_{n},
\end{displaymath}
$I_{n}$ being the $n\times n$ identity matrix. The relation above is equivalent to
\begin{displaymath}
U^{\ast} U = I_{n}.
\end{displaymath}
If $U\in U(n,\mathbb{H})$,
\begin{displaymath}
\det(\mathtt{C}(U))=1,
\end{displaymath}
and $\mathtt{C}(U)\in SU(2n)$. The image of 
$U(n,\mathbb{H})$ by $\mathtt{C}$ is the compact symplectic group
$Sp(n)$, a subgroup of $SU(2n)$. 

The set of $n\times n$ quaternionic Hermitian matrices
$\mathcal{H}_{n}(\mathbb{H})$ is composed of matrices
$M\in\mathcal{M}_{n}(\mathbb{H})$ satisfying
\begin{displaymath}
M^{\ast}=M.
\end{displaymath}
A matrix $M\in\mathcal{M}_{n}(\mathbb{H})$ 
is quaternionic Hermitian if and only if $\mathtt{C}(M)$ is
complex Hermitian. The diagonal entries of a quaternionic Hermitian matrix are real. Given $M\in \mathcal{H}_{n}(\mathbb{H})$, there exists
$U\in U(n,\mathbb{H})$ such that $U^{\ast} M U$ is diagonal with real entries:
\begin{displaymath}
U^{\ast} M U =
\operatorname{Diag}(\lambda_{1},\lambda_{2},\dots,\lambda_{n}),
\end{displaymath}
with $\lambda_{1}\geq\lambda_{2}\geq\dots\geq\lambda_{n}\in\mathbb{R}$.
The family $\lambda_{1},\lambda_{2},\dots,\lambda_{n}$ is uniquely determined. This is also the family of ordered eigenvalues of
$\mathtt{C}(M)$, but the multiplicities have to be doubled. 
$\lambda_{1},\lambda_{2},\dots,\lambda_{n}$ are the right eigenvalues of
$M$, and form the right spectrum of $M$, i.e. the set of 
$\lambda\in\mathbb{H}$, for which the equation
\begin{displaymath}
Mx = x\lambda
\end{displaymath}
has a non-zero solution. There is also a notion of left spectrum, corresponding to the equation
$Mx = \lambda x$.
But the right and the left spectra do not necessarily coincide, even for quaternionic Hermitian matrices. For details on eigenvalues of quaternionic matrices we refer to
\cite{Zhang97QuaternionMatrices,HandbookQuaternions14}.
For $M\in\mathcal{H}_{n}(\mathbb{H})$,
the trace of $M$ also equals
\begin{displaymath}
\Tr(M)=\Re(\Tr(M))=\sum_{i=1}^{n}\lambda_{i}.
\end{displaymath}

\subsection{BFS-Dynkin isomorphism}
\label{ssecDynkin}

Let $\mathcal{G}=(V,E)$ be a finite undirected connected graph. 
We do not allow multiple edges or self-loops. 
Edges $\{x,y\}\in E$ are endowed with conductances $C(x,y)=C(y,x)>0$. 
There also a not identically zero killing measure
$(\kappa(x))_{x\in V}$, with $\kappa(x)\geq 0$. 
The graph $\mathcal{G}$ will be further referred to as
an electrical network. 
Let $(X_{t})_{t\geq 0}$ be the Markov jump process to nearest neighbors with jump rates given by the conductances. 
$(X_{t})_{t\geq 0}$ is also killed by $\kappa$. 
Let $\zeta\in (0,+\infty]$ be the first time 
$(X_{t})_{t\geq 0}$ gets killed by $\kappa$. 

Let $(G(x,y))_{x,y\in V}$ be the Green's function:
\begin{displaymath}
G(x,y)=G(y,x)=
\E\Big[\int_{0}^{\zeta}
\1_{\{X_{t}=y\}} dt
\big\vert X_{0}=x
\Big].
\end{displaymath}
Let $p_{t}(x,y)$ be the transition probabilities of
$(X_{t})_{0\leq t <\zeta}$. Then
$p_{t}(x,y)=p_{t}(y,x)$ and
\begin{displaymath}
G(x,y)=\int_{0}^{+\infty}p_{t}(x,y) dt.
\end{displaymath}
Let $\mathbb{P}_{t}^{x,y}$ be the bridge probability measure from $x$ to $y$, where one conditions by $t<\zeta$. 
Let $\mu^{x,y}$ be the following measure on paths from $x$ to $y$ in finite time:
\begin{equation}
\label{Eqmuxy}
\mu^{x,y}(d\gamma)=\int_{0}^{+\infty}
\mathbb{P}_{t}^{x,y}(d\gamma)
p_{t}(x,y) dt.
\end{equation}
The measure $\mu^{x,y}$ has total mass $G(x,y)$. 
The image of $\mu^{x,y}$ by time reversal is $\mu^{y,x}$.

Given $x_{1},x_{2},\dots,x_{2r}\in V$ and
$\textbf{p}=\{\{a_{1},b_{1}\},\dots,\{a_{r},b_{r}\}\}$
a partition in pairs of
$\{1,\dots, 2r\}$, 
$\mu^{x_{1},x_{2},\dots, x_{2r}}_{\textbf{p}}
(d\gamma_{1},\dots,d\gamma_{r})$ will denote the following product measure on $r$-tuples of paths:
\begin{displaymath}
\mu^{x_{1},x_{2},\dots, x_{2r}}_{\textbf{p}}
(d\gamma_{1},\dots,d\gamma_{r})=
\prod_{i=1}^{r}\mu^{x_{a_{i}},x_{b_{i}}}(d\gamma_{i}).
\end{displaymath}
We will sometimes, in particular in Section \ref{SubSecMatrixTwist}, use the convention that
$a_{i}<b_{i}$.
The order of the pairs in $\textbf{p}$ will not be important.

In general, for a path $\gamma$ and
$x\in V$, $L(\gamma)$ will denote the occupation field of
$\gamma$,
\begin{displaymath}
L(\gamma)(x)=\int_{0}^{T(\gamma)}\1_{\{\gamma(t)=x\}} dt,
\end{displaymath}
where $T(\gamma)$ is the life-time of the path.

The (scalar real) Gaussian free field (GFF) $(\phi(x))_{x\in V}$ will denote here the centered Gaussian process with covariance 
\begin{displaymath}
\E[\phi(x)\phi(y)] = G(x,y).
\end{displaymath}
The distribution of $(\phi(x))_{x\in V}$ is given by
\begin{equation}
\label{EqDensityGFF}
\dfrac{1}{((2\pi)^{\operatorname{Card}(V)} \det G)^{\frac{1}{2}}}
\exp\Big(
-\dfrac{1}{2}\sum_{x\in V}\kappa(x)\varphi(x)^{2}
-\dfrac{1}{2}\sum_{\{x,y\}\in E}C(x,y)(\varphi(y)-\varphi(x))^{2}
\Big) \prod_{x\in V} d\varphi(x).
\end{equation}

The following isomorphism relates the square of the GFF 
$(\phi(x)^{2})_{x\in V}$ and the occupation fields of paths under the measures $\mu^{x,y}$. It first appears in the work of
Brydges, Fröhlich and Spencer \cite{BFS82Loop}
(see also \cite{Frohlich82Triv,BFSok83one,BFSok83two}) 
and then in that of Dynkin
\cite{Dynkin1984Isomorphism,Dynkin1984IsomorphismPresentation}
(see also \cite{Dynkin1984PolynomOccupField}). 
It is also related to earlier works of Symanzik \cite{Symanzik65Scalar,Symanzik66Scalar,Symanzik1969QFT}.
For more on isomorphism theorems, see
\cite{MarcusRosen2006MarkovGaussianLocTime,
Sznitman2012LectureIso}.

\begin{thm}[Brydges-Fröhlich-Spencer \cite{BFS82Loop},
Dynkin
\cite{Dynkin1984Isomorphism,Dynkin1984IsomorphismPresentation}]
\label{ThmIsoDynkin}
Let $r\in\mathbb{N}\setminus \{0\}$, 
$x_{1},x_{2},\dots,x_{2r}\in V$
and $F$ a bounded measurable function 
$\mathbb{R}^{V} \rightarrow \mathbb{R}$. Then
\begin{displaymath}
\mathbb{E}\Big[
\prod_{k=1}^{2r}\phi(x_{k})
F(\phi^{2}/2)
\Big]=
\sum_{\substack{\textbf{p} \text{ partition}
\\\text{of } \{1,\dots, 2r\}
\\\text{in pairs}
}} 
\int\displaylimits_{\gamma_{1},\dots \gamma_{r}}
\E\Big[
F(\phi^{2}/2+L(\gamma_{1})+\dots + L(\gamma_{r}))
\Big]
\mu^{x_{1},x_{2},\dots, x_{2r}}_{\textbf{p}}
(d\gamma_{1},\dots,d\gamma_{r}).
\end{displaymath}
\end{thm}

\subsection{Connections, gauge equivalence and Wilson loops}
\label{SubSecConx}

Let $n\in\N$, $n\geq 2$. 
Let $\mathbb{U}$ be the group of either $n\times n$ orthogonal matrices $O(n)$, or unitary matrices $U(n)$, or quaternionic unitary matrices $U(n,\mathbb{H})$.
We consider that each undirected edge in
$E$ consists of two directed edges of opposite direction.
We consider a family of matrices in $\mathbb{U}$,
$(U(x,y))_{\{x,y\}\in E}$, with
\begin{displaymath}
U(y,x)=U(x,y)^{\ast}=U(x,y)^{-1},
~~\forall \{x,y\}\in E.
\end{displaymath}
$(U(x,y))_{\{x,y\}\in E}$ is our \textit{connection} on the vector bundle with base space $\mathcal{G}$ 
and fiber respectively $\R^{n}$, $\C^{n}$ or $\mathbb{H}^{n}$.

Given a nearest-neighbor oriented discrete path
$\gamma=(y_{1},y_{2},\dots,y_{j})$, the
\textit{holonomy} of $U$ along $\gamma$ is the product
\begin{displaymath}
\mathsf{hol}^{U}(\gamma)=
U(y_{1},y_{2})U(y_{2},y_{3})\dots U(y_{j-1},y_{j}).
\end{displaymath}
If the path $\gamma$ is a nearest-neighbor path parametrized by continuous time, and does only a finite number of jumps, the holonomy
$\mathsf{hol}^{U}(\gamma)$ is defined as the holonomy along the discrete skeleton of $\gamma$. 
We will denote by $\overleftarrow{\gamma}$ 
the time-reversal of a path $\gamma$. 
We have that
\begin{equation}
\label{EqHolAst}
\mathsf{hol}^{U}(\overleftarrow{\gamma})=
\mathsf{hol}^{U}(\gamma)^{\ast}=
\mathsf{hol}^{U}(\gamma)^{-1}.
\end{equation}

Given a nearest-neighbor oriented discrete closed path
(i.e. a loop) $\gamma=(y_{1},y_{2},\dots,y_{j})$,
with $y_{j}=y_{1}$, we will consider the observable
\begin{displaymath}
\Tr(\mathsf{hol}^{U}(\gamma))
\end{displaymath}
in the orthogonal and unitary case, and
\begin{displaymath}
\Re(\Tr(\mathsf{hol}^{U}(\gamma)))=
\dfrac{1}{2}\Tr(\mathtt{C}(\mathsf{hol}^{U}(\gamma)))
\end{displaymath}
in the quaternionic unitary case. Such observables are called
\textit{Wilson loops} \cite{Wilson74Quarks}.
Note that the Wilson loop observable does not depend on where 
the loop $\gamma$ is rooted.
Indeed, if 
$\tilde{\gamma}$ is the loop visiting
$(y_{i},\dots,y_{j},y_{1},\dots,y_{i-1},y_{i})$
$(i\in\{2,\dots, j\})$, and if
$\gamma'$ is the path visiting
$(y_{1},\dots,y_{i})$ then
\begin{displaymath}
\mathsf{hol}^{U}(\tilde\gamma)=
\mathsf{hol}^{U}(\gamma')^{-1}
\mathsf{hol}^{U}(\gamma)
\mathsf{hol}^{U}(\gamma').
\end{displaymath}

Given another family of matrices in $\mathbb{U}$,
$(\mathfrak{U}(x))_{x\in V}$, this time on top of vertices, it induces a \textit{gauge transformation} on the connection $U$:
\begin{displaymath}
(U(x,y))_{\{x,y\}\in E}
\longmapsto 
(\mathfrak{U}(x)^{-1}U(x,y)\mathfrak{U}(y))_{\{x,y\}\in E}.
\end{displaymath}
Two connections related by a gauge transformation are said to be
\textit{gauge equivalent}. 
A connection is \textit{trivial} if it is gauge equivalent to the identity connection. 
A criterion for triviality is that along any nearest-neighbor loop
$\gamma=(y_{1},y_{2},\dots,y_{j})$, with $y_{j}=y_{1}$,
\begin{displaymath}
\mathsf{hol}^{U}(\gamma) = I_{n}.
\end{displaymath}
In general, any two gauge equivalent connections have the same Wilson loop observables. The converse is also true (but non-obvious): the collection of all possible Wilson loop observables characterizes a connection up to gauge transformations
\cite{Giles81Reconstruct,Sengupta94GaugeInvariant,
TLevy04WilsonLoops}.

\subsection{BFS-Dynkin isomorphism for the Gaussian free field twisted by a connection}
\label{SubSecCovSym}

In \cite{KasselLevy16CovSym} Kassel and Lévy introduced the 
vector-valued GFF twisted by an orthogonal/unitary connection, and generalized the isomorphisms with random walks to this case. Here we will do a less abstract, more computational presentation of the same object. Kassel and Lévy's isomorphisms rely on a
covariant Feynman-Kac formula (\cite{BFSei79Higgs} and 
Theorem 3.1 in \cite{KasselLevy16CovSym}).

Let us consider on top of the electrical network
$\mathcal{G}=(V,E)$ an orthogonal connection
$(U(x,y))_{\{x,y\}\in E}$, $U(x,y)\in O(n)$.
The Green's function $G^{U}$ associated to the connection $U$
is a function from $V\times V$ to 
$\mathcal{M}_{n}(\mathbb{R})$ (i.e. the $n\times n$ matrices with real entries), with the entries given by
\begin{displaymath}
G^{U}_{ij}(x,y)= 
\int_{\gamma}\mathsf{hol}^{U}_{ij}(\gamma)\mu^{x,y}(d\gamma),
~~x,y\in V,~i,j\in\{1,\dots,n\},
\end{displaymath}
where the measure on paths $\mu^{x,y}(d\gamma)$ is given by
\eqref{Eqmuxy}. Since the image of 
$\mu^{x,y}$ by time reversal is
$\mu^{y,x}$, and because of 
\eqref{EqHolAst}, we have that
\begin{displaymath}
G^{U}_{ij}(x,y) = G^{U}_{ji}(y,x),\qquad
G^{U}_{ij}(x,x) = G^{U}_{ji}(x,x).
\end{displaymath}
i.e. $G^{U}(x,y)^{\mathsf{T}} = G^{U}(y,x)$
and $G^{U}(x,x)$ is symmetric.
One can see $G^{U}$ as a symmetric linear operator on
$(\R^{n})^{V}$. It is positive definite
(see Proposition 2.15 in \cite{KasselLevy16CovSym}).
We will denote by $\det G^{U}$ the determinant of this operator.

The $\R^{n}$-valued Gaussian free field on $\mathcal{G}$ twisted by the connection $U$ is a random Gaussian function
$\widehat{\phi}: V \rightarrow \R^{n}$ 
($\widehat{\phi}(x)=(\widehat{\phi}_{1}(x),\dots, 
\widehat{\phi}_{n}(x))$) 
with the distribution given by
\begin{displaymath}
\dfrac{1}{Z_{\rm GFF}^{U}}
\exp\Big(
-\dfrac{1}{2}\sum_{x\in V}\kappa(x)\Vert \widehat{\varphi}(x)\Vert^{2}
-\dfrac{1}{2}\sum_{\{x,y\}\in E}
C(x,y)\Vert \widehat{\varphi}(x)-U(x,y)\widehat{\varphi}(y)\Vert^{2}
\Big)
\prod_{x\in V}\prod_{i=1}^{n}d\widehat{\varphi}(x)_{i},
\end{displaymath}
where $\Vert\cdot\Vert$ is the usual
$L^{2}$ norm on $\R^{n}$ and
\begin{displaymath}
Z_{\rm GFF}^{U}=
((2\pi)^{n \operatorname{Card}(V)} \det G^{U})^{\frac{1}{2}}.
\end{displaymath}
Note that if $\{x,y\}\in E$,
\begin{displaymath}
\Vert \widehat{\varphi}(x)-U(x,y)\widehat{\varphi}(y)\Vert^{2}
=
\Vert \widehat{\varphi}(y)-U(y,x)\widehat{\varphi}(x)\Vert^{2}.
\end{displaymath}

We have that $\E[\widehat{\phi}]\equiv 0$.
As for the covariance structure, we have
(see Proposition 4.1 in \cite{KasselLevy16CovSym}):
\begin{displaymath}
\E[\widehat{\phi}_{i}(x)\widehat{\phi}_{j}(y)]=
G^{U}_{ij}(x,y).
\end{displaymath}
If $(\mathfrak{U}(x))_{x\in V}$ is a gauge transformation, then 
$(\mathfrak{U}(x)\widehat{\phi}(x))_{x\in V}$ is the Gaussian free field related to the connection
$(\mathfrak{U}(x)^{-1}U(x,y)\mathfrak{U}(y))_{\{x,y\}\in E}$.
In particular, if the connection $U$ is trivial, 
the field $\widehat{\phi}$ can be reduced to $n$ i.i.d. copies of the scalar GFF with distribution \eqref{EqDensityGFF}.

In \cite{KasselLevy16CovSym}, Theorems 5.1 and 7.3, Kassel and Lévy gave a BFS-Dynkin-type isomorphism for GFFs twisted by connections. 

\begin{thm}[Kassel-Lévy
\cite{KasselLevy16CovSym}]
\label{ThmIsoKL}
Let $r\in\mathbb{N}\setminus \{0\}$, 
$x_{1},x_{2},\dots,x_{2r}\in V$,
$J(1),J(2),\dots, J(2r)\in \{1,\dots,n\}$
and $F$ a bounded measurable function 
$\mathbb{R}^{V} \rightarrow \mathbb{R}$. Then
\begin{multline*}
\mathbb{E}\Big[
\prod_{k=1}^{2r}\widehat{\phi}_{J(k)}(x_{k})
F(\Vert\widehat{\phi}\Vert^{2}/2)
\Big]\\
=
\sum_{\substack{\text{partitions of}
\\\{1,\dots, 2r\}
\\\text{in pairs}
\\\{\{a_{1},b_{1}\},\dots,\{a_{r},b_{r}\}\}}} 
\int\displaylimits_{\gamma_{1},\dots \gamma_{r}}
\E\Big[
F(\Vert\widehat{\phi}\Vert^{2}/2
+L(\gamma_{1})+\dots + L(\gamma_{r}))
\Big]
\prod_{i=1}^{r}
\mathsf{hol}^{U}_{J(a_{i})J(b_{i})}(\gamma_{i})
\mu^{x_{a_{i}},x_{b_{i}}}(d\gamma_{i}),
\end{multline*}
where the sum runs over the $(2r)!/(2^{r} r!)$ partitions of
$\{1,\dots, 2r\}$ in pairs.
\end{thm}

Next we rewrite slightly the isomorphism above.
This will be used in Section \ref{SecProofs},
in the proofs of Lemmas \ref{LemMi} and \ref{LemMiHol}.
Let $(\widehat{X}^{(i)})_{i\geq 1}$ be an i.i.d. family of random Gaussian vectors with $n$ components, 
following the law $\mathcal{N}(0,I_{n})$. 

\begin{lemma}
\label{LemIsoKLrew}
Let $r\in\mathbb{N}\setminus \{0\}$, 
$x_{1},x_{2},\dots,x_{2r}\in V$,
$J(1),J(2),\dots, J(2r)\in \{1,\dots,n\}$
and $F$ a bounded measurable function 
$\mathbb{R}^{V} \rightarrow \mathbb{R}$. Then
\begin{multline*}
\mathbb{E}\Big[
\prod_{k=1}^{2r}\widehat{\phi}_{J(k)}(x_{k})
F(\Vert\widehat{\phi}\Vert^{2}/2)
\Big]\\
=\sum_{\substack{\text{partitions of}
\\\{1,\dots, 2r\}
\\\text{in pairs}
\\\{\{a_{1},b_{1}\},\dots,\{a_{r},b_{r}\}\}}} 
\int\displaylimits_{\gamma_{1},\dots \gamma_{r}}
\E\Big[
F(\Vert\widehat{\phi}\Vert^{2}/2
+L(\gamma_{1})+\dots + L(\gamma_{r}))
\Big]\times
\\
\times\E\Big[
\prod_{i=1}^{r}
\Big(\mathsf{hol}^{U}(\gamma_{i})
\widehat{X}^{(i)}\Big)_{J(a_{i})}
\widehat{X}^{(i)}_{J(b_{i})}
\Big]
\prod_{i=1}^{r}
\mu^{x_{a_{i}},x_{b_{i}}}(d\gamma_{i}),
\end{multline*}
where the sum runs over the $(2r)!/(2^{r} r!)$ partitions of
$\{1,\dots, 2r\}$ in pairs.
\end{lemma}

\begin{proof}
Indeed,
\begin{displaymath}
\E\Big[
\prod_{i=1}^{r}
\Big(\mathsf{hol}^{U}(\gamma_{i})
\widehat{X}^{(i)}\Big)_{J(a_{i})}
\widehat{X}^{(i)}_{J(b_{i})}
\Big]=
\prod_{i=1}^{r}
\mathsf{hol}^{U}_{J(a_{i})J(b_{i})}(\gamma_{i}).
\qedhere
\end{displaymath}
\end{proof}

\subsection{Ribbon graphs and surfaces}
\label{SubSecRibbon}

Here we describe the ribbon graphs and the related two-dimensional surfaces. For more details, we refer to
\cite[Sections~2.2, 2.3]{Eynard16CountingSurfaces},
\cite[Chapter~2]{EynardKimuraRibault18RandomMatrices},
\cite[Sections~3.2, 3.3]{LandoZvonkine04GraphsSurfaces},
\cite{Zvonkin97MatrixIntegrals},
and 
\cite[Section~3.3]{MoharThomassen01GraphsSurf}.

Let $\nu=(\nu_{1},\nu_{2},\dots,\nu_{m})$, where
$m\geq 1$, and for all
$l\in\{1,2,\dots,m\}$,
$\nu_{l}\in\mathbb{N}\setminus\{0\}$.
We will denote
\begin{displaymath}
m(\nu) = m, \qquad
\vert\nu\vert = \sum_{l=1}^{m(\nu)}\nu_{l}.
\end{displaymath}
We will assume that $\vert\nu\vert$ is even.

Given $\nu$ as above, we consider $m(\nu)$ vertices, where each vertex has adjacent \textit{ribbon half-edges}: 
$\nu_{1}$ half-edges for the first vertex, $\nu_{2}$
for the second, etc. 
A ribbon half-edge is a two-dimensional object and carries an orientation. Also, the ribbon half-edges around each vertex are ordered in a cyclic way.
The ribbon half-edges are numbered from $1$ to 
$\vert\nu\vert$. 
See Figure \ref{FigRibbonHE} for an illustration with 
$\nu=(4,3,1)$.

\begin{figure}[ht]
\centering
\includegraphics[scale=0.8]{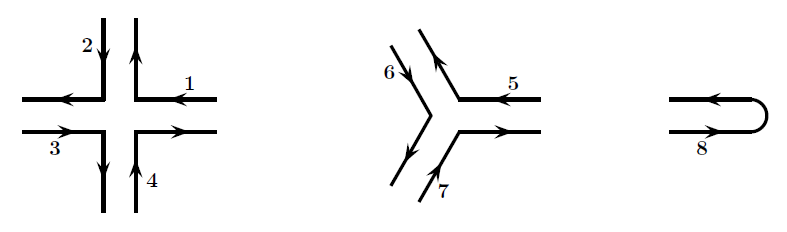}
\caption{Ribbon half-edges in the case of $\nu=(4,3,1)$.}
\label{FigRibbonHE}
\end{figure}

Since the total number of half-edges, $\vert\nu\vert$, is even, one can pair them to obtain a \textit{ribbon graph} (not necessarily connected), with $m(\nu)$ vertices and $\vert\nu\vert/2$ \textit{ribbon edges}. 
Each time we pair two half-edges, we can glue the corresponding ribbons in two different ways. Either the orientations of the two ribbon half-edges match, or are opposite. In the first case we get a \textit{straight ribbon edge}, in the second a \textit{twisted ribbon edge}. 
See Figure \ref{FigTwoGlueings}.
We call such a pairing of ribbon half-edges that keeps straight or twists the ribbons a \textit{ribbon pairing}.
Let $\mathcal{R}_{\nu}$ be the set of all possible ribbon pairings associated to $\nu$. The number of different ribbon pairings is
\begin{displaymath}
\operatorname{Card}(\mathcal{R}_{\nu})=
\dfrac{\vert\nu\vert !}
{2^{\vert\nu\vert/2}
(\vert\nu\vert/2)!} 2^{\vert\nu\vert/2} =
\dfrac{\vert\nu\vert !}
{(\vert\nu\vert/2)!}.
\end{displaymath}
Figure \ref{FigRibbonGraphstr} displays an example of a ribbon pairing with only straight edges, and Figure
\ref{FigRibbonGraph} an example with both straight and twisted edges.

\begin{figure}[ht]
\centering
\includegraphics[scale=0.8]{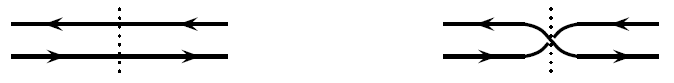}
\caption{A straight ribbon edge on the left and a twisted ribbon edge on the right.}
\label{FigTwoGlueings}
\end{figure}

\begin{figure}[ht]
\centering
\includegraphics[scale=0.8]{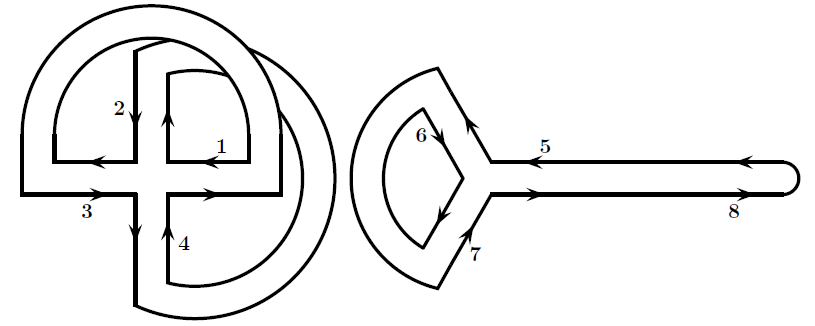}
\caption{A ribbon pairing in the case of $\nu=(4,3,1)$ with only straight edges.
The induced partition in pairs is
$\textbf{p}_{\nu}(\rho)=
\{\{1,3\},\{2,4\},\{5,8\},\{6,7\}\}$.}
\label{FigRibbonGraphstr}
\end{figure}

\begin{figure}[ht]
\centering
\includegraphics[scale=0.8]{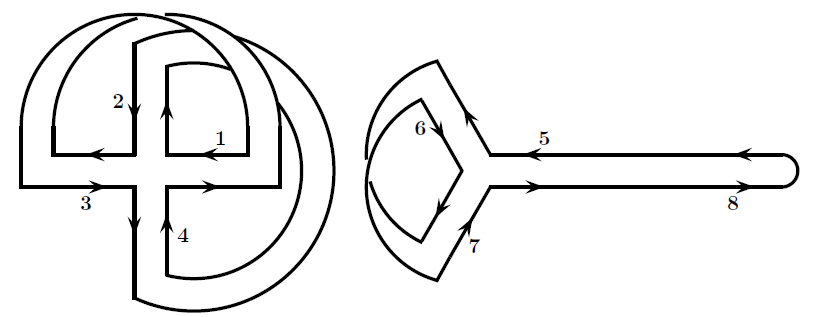}
\caption{A ribbon pairing in the case of $\nu=(4,3,1)$ with straight and twisted edges.
The induced partition in pairs $\textbf{p}_{\nu}(\rho)$ is
the same as on Figure \ref{FigRibbonGraphstr}.}
\label{FigRibbonGraph}
\end{figure}

A ribbon pairing $\rho\in\mathcal{R}_{\nu}$ induces a partition in pairs of $\{1,\dots,\vert\nu\vert\}$, denoted  
$\textbf{p}_{\nu}(\rho)$. 
The pairs correspond to the labels of ribbon half-edges associated into an edge. 
Conversely, given $\textbf{p}$ a partition in pairs of
$\{1,\dots,\vert\nu\vert\}$,
$\mathcal{R}_{\nu,\textbf{p}}$  will denote
the subset of $\mathcal{R}_{\nu}$
made of all ribbon pairings $\rho$ such that
$\textbf{p}_{\nu}(\rho)=\textbf{p}$
($\operatorname{Card}(\mathcal{R}_{\nu,\textbf{p}})
=2^{\vert\nu\vert/2}$).

Given a ribbon pairing $\rho\in\mathcal{R}_{\nu}$, one can see the corresponding ribbon graph as a two-dimensional compact bordered surface (not necessarily connected). 
Let $f_{\nu}(\rho)$ denote the number of the connected components of the boundary, 
that is to say the number of distinct cycles formed by the borders of ribbons. 
On Figure \ref{FigRibbonGraphstr}, 
$f_{\nu}(\rho)=3$, and on Figure \ref{FigRibbonGraph}, 
$f_{\nu}(\rho)=2$.
Then, one can glue along each connected component of the boundary a disk ($f_{\nu}(\rho)$ disks in total), 
and obtain in this way
a two-dimensional compact surface (not necessarily connected) without boundary. 
We will denote it 
$\Sigma_{\nu}(\rho)$, and consider it up to diffeomorphisms.
On the example of Figure \ref{FigRibbonGraphstr},  
$\Sigma_{\nu}(\rho)$ has two connected components,
a torus on the left and a sphere on the right.
On the example of 
Figure \ref{FigRibbonGraph},  
$\Sigma_{\nu}(\rho)$ has again two connected components,
a Klein bottle on the left and a projective plane on the right. 
Observe that if all the edges are straight, the surfaces that appear are orientable. Let $\chi_{\nu}(\rho)$ denote the Euler's characteristic of $\Sigma_{\nu}(\rho)$. According to
Euler's formula,
\begin{displaymath}
\chi_{\nu}(\rho)= m(\nu) - \dfrac{\vert\nu\vert}{2}
+f_{\nu}(\rho).
\end{displaymath}

Next we introduce additional combinatorial objects related to the ribbon pairings. We will consider tuples 
$(k_{1},\mathtt{s}_{1},k_{2},\mathtt{s}_{2},
\dots,k_{j},\mathtt{s}_{j})$, where
$j\in\mathbb{N}\setminus\{0\}$, each of the
$k_{i}$ is in $\mathbb{N}\setminus\{0\}$, and each of the
$\mathtt{s}_{i}$ is one of the three abstract symbols
$\rightarrow$, $\leftarrow$ or $\doteq$. We will endow such tuples by an equivalence relation $\approx$ generated by the following rules.
\begin{itemize}
\item Cyclic permutation: for any $i\in\{2,\dots,j\}$,
$(k_{i},\mathtt{s}_{i},\dots, k_{j},\mathtt{s}_{j},
k_{1},\mathtt{s}_{1},\dots,k_{i-1},\mathtt{s}_{i-1})$ is identified to
$(k_{1},\mathtt{s}_{1},k_{2},\mathtt{s}_{2},
\dots,k_{j},\mathtt{s}_{j})$.
\item Reversal of the direction: 
$(k_{j},\mathtt{r}(\mathtt{s}_{j}),\dots,
k_{2},\mathtt{r}(\mathtt{s}_{2}),k_{1},\mathtt{r}(\mathtt{s}_{1}))$
is identified to 
\\$(k_{1},\mathtt{s}_{1},k_{2},\mathtt{s}_{2},
\dots,k_{j},\mathtt{s}_{j})$, where
$\mathtt{r}(\rightarrow)$ is $\leftarrow$,
$\mathtt{r}(\leftarrow)$ is $\rightarrow$, and
$\mathtt{r}(\doteq)$ is $\doteq$.
\end{itemize}
For lack of a better name, we will call \textit{trails} the equivalence classes of $\approx$.

Given a ribbon pairing $\rho\in\mathcal{R}_{\nu}$, we will associate to $\rho$ a set $\mathcal{T}_{\nu}(\rho)$ made of
$f_{\nu}(\rho)$ trails, one per each boundary cycle in the ribbon pairing.
One starts on such a boundary cycle in an arbitrary place, 
and travels along it in any of the two directions. 
Then one successively visits ribbon half-edges with labels 
$k_{1},k_{2},\dots,k_{j}$ and then returns to the half-edge $k_{1}$.
One can go from the half-edge $k_{i}$ to the half-edge
$k_{i+1}$ either by following a gluing, and we will denote this
$k_{i}\doteq k_{i+1}$, or by going through a vertex. In the latter case, one either does a turn clockwise, and we will denote this
$k_{i}\rightarrow k_{i+1}$, or counterclockwise, and we will denote this $k_{i}\leftarrow k_{i+1}$. A special rule is applied if the vertex has only one outgoing half-edge, one just makes the arrows in the trail and on the picture match. This is how a trail is obtained. Note that by construction, there is an alternation between on one hand
$\doteq$, and on the other hand 
$\rightarrow$ or $\leftarrow$. In the example of 
Figure \ref{FigRibbonGraphstr}, there are three trails:
\begin{equation}
\label{EqTrail_str}
(1,\rightarrow,2,\doteq,4,\rightarrow,1,\doteq,3,\rightarrow,4,
\doteq,2,\rightarrow,3,\doteq),
~~ 
(5,\rightarrow,6,\doteq,7,\rightarrow,5,\doteq,8,\rightarrow,8,\doteq),
~~ (6,\rightarrow,7,\doteq).
\end{equation}
In the example of 
Figure \ref{FigRibbonGraph}, there are two trails:
\begin{displaymath}
(1,\rightarrow,2,\doteq,4,\rightarrow,1,\doteq,3,\leftarrow,2,
\doteq,4,\leftarrow,3,\doteq),
~~ 
(5,\rightarrow,6,\doteq,7,\leftarrow,6,\doteq,7,\rightarrow,5,\doteq,8,\rightarrow,8,\doteq).
\end{displaymath}

We will also consider \textit{oriented trails}. Like the (unoriented) trails, they can contain positive integer numbers and symbols $\rightarrow$ and $\doteq$, but not the symbol
$\leftarrow$. In the oriented trails we quotient by cyclic permutations, but not by the reversal of direction. We will associate oriented trails to ribbon pairing that contain only straight edges. Given $\rho\in\mathcal{R}_{\nu}$ with only straight edges, $\overrightarrow{\mathcal{T}}_{\nu}(\rho)$ will be a set of
$f_{\nu}(\rho)$ oriented trails, 
one per each boundary cycle where one follows the cycle in the direction of the arrows (clockwise).
The oriented trails corresponding to Figure \ref{FigRibbonGraphstr} are given by \eqref{EqTrail_str}.

\subsection{One matrix integrals and topological expansion}
\label{SubSecOneMatrix}
$E_{\beta,n}$ will denote $\mathcal{S}_{n}(\R)$, the space of real symmetric matrices, for 
$\beta=1$, $\mathcal{H}_{n}(\mathbb{C})$, 
the space of complex Hermitian matrices, for
$\beta=2$, and $\mathcal{H}_{n}(\mathbb{H})$,
the space of quaternionic Hermitian matrices,
 for $\beta=4$.
 \begin{displaymath}
\dim E_{\beta=1,n} = \dfrac{(n+1)n}{2},
\qquad
\dim E_{\beta=2,n} = n^{2},
\qquad
\dim E_{\beta=4,n} = 2n^{2} -n.
\end{displaymath}
$E_{\beta,n}$ is endowed with the real inner product
\begin{displaymath}
(M,M')\mapsto \Re(\Tr(M M')).
\end{displaymath}

The Gaussian Orthogonal Ensemble GOE$(n)$,
the Gaussian Unitary Ensemble GUE$(n)$ and 
Gaussian Symplectic Ensemble GSE$(n)$
are Gaussian probability measures on
$E_{\beta,n}$, with $\beta=1$ for the GOE$(n)$,
$\beta=2$ for the GUE$(n)$ and $\beta=4$ for the GSE$(n)$. We will use the usual notation G$\beta$E$(n)$. The density with respect to the Lebesgue measure on $E_{\beta,n}$ is given by
\begin{equation}
\label{EqGbetaEmat}
\dfrac{1}{Z_{\beta,n}}e^{-\frac{1}{2}\Tr(M^{2})}.
\end{equation}
The distribution of the ordered family of eigenvalues
$\lambda_{1}\geq \lambda_{2}\geq\dots \geq \lambda_{n}$ of 
G$\beta$E$(n)$ is given by
\begin{displaymath}
\dfrac{1}{Z_{\beta,n}^{\rm ev}}
\1_{\{\lambda_{1}\geq \lambda_{2}\geq\dots \geq \lambda_{n}\}}
\prod_{1\leq i<j\leq n}(\lambda_{i}-\lambda_{j})^{\beta}
e^{-\frac{1}{2}(\lambda_{1}^{2}+\dots + \lambda_{n}^{2})}
d\lambda_{1}\dots d\lambda_{n}.
\end{displaymath}
For more on random matrices see \cite{Mehta04RandomMatrices}.

Let $\nu=(\nu_{1},\nu_{2},\dots,\nu_{m(\nu)})$, 
where for all
$l\in\{1,2,\dots,m(\nu)\}$,
$\nu_{l}\in\mathbb{N}\setminus\{0\}$, and
\begin{displaymath}
\vert \nu\vert= \sum_{l=1}^{m(\nu)}\nu_{l}
\end{displaymath}
is even. Next we recall the expressions for the matrix integrals
\begin{equation}
\label{EqMom}
\langle
\prod_{l=1}^{m(\nu)}
\Tr(M^{\nu_{l}})
\rangle_{\beta,n}=
\dfrac{1}{Z_{\beta,n}}
\int_{E_{\beta,n}}
\Big(\prod_{l=1}^{m(\nu)}
\Tr(M^{\nu_{l}})\Big)
e^{-\frac{1}{2}\Tr(M^{2})} dM,
\qquad \beta\in\{1,2,4\}.
\end{equation}
Note that if $\vert \nu\vert$ is odd, the above integrals are
zero.
The expression for \eqref{EqMom} is a polynomial in $n$, 
with powers $n^{f_{\nu}(\rho)}$, where 
$\rho\in\mathcal{R}_{\nu}$ are ribbon pairings associated to 
$\nu$. Since $f_{\nu}(\rho)$ can be expressed using Euler's characteristic of surfaces, the expression for \eqref{EqMom} is often referred to as \textit{topological expansion}.
The expansion for complex Hermitian and real symmetric matrices
appears in \cite{BIPZ78PlanarDiagrams}
(see also \cite{BIZ80GraphicalEnumeration}).
The combinatorics for quaternionic Hermitian matrices are given in
\cite{MulaseWaldron03GSE} (see also \cite{BrycPierce09GSE}). 
For more on the topological expansion, we  also refer to
\cite[Chapter~2]{Eynard16CountingSurfaces},
\cite[Chapter~2]{EynardKimuraRibault18RandomMatrices},
\cite[Chapter~3]{LandoZvonkine04GraphsSurfaces},
and
\cite{Zvonkin97MatrixIntegrals}. 

Given a ribbon pairing $\rho\in\mathcal{R}_{\nu}$, we  associate to it a weight $w_{\nu,\beta}(\rho)$ 
depending on $\beta$:
\begin{displaymath}
w_{\nu,\beta=1}(\rho)=\dfrac{1}{2^{\vert\nu\vert/2}},
~~
w_{\nu,\beta=2}(\rho)=
1_{\rho \text{ has only straight edges}},
~~
w_{\nu,\beta=4}(\rho)=(-2)^{\chi_{\nu}(\rho)}
2^{-2 m(\nu)+\vert\nu\vert/2}.
\end{displaymath}
In all three cases $\beta\in\{1,2,4\}$,
for every $\textbf{p}$ partition in pairs of
$\{1,\dots,\vert\nu\vert\}$,
\begin{displaymath}
\sum_{\rho\in\mathcal{R}_{\nu,\textbf{p}}}w_{\nu,\beta}(\rho)=
1.
\end{displaymath}

\begin{thm}[Brézin-Itzykson-Parisi-Zuber \cite{BIPZ78PlanarDiagrams},
Mulase-Waldron \cite{MulaseWaldron03GSE}]
\label{ThmTopoExp1Matrix}
For $\beta\in\{1,2,4\}$ and $\vert\nu\vert$ even, the value of the matrix integral
$\langle
\prod_{l=1}^{m(\nu)}
\Tr(M^{\nu_{l}})
\rangle_{\beta,n}$ \eqref{EqMom} is given by
\begin{equation}
\label{EqTopoExpansion}
\langle
\prod_{l=1}^{m(\nu)}
\Tr(M^{\nu_{l}})
\rangle_{\beta,n}=
\sum_{\rho\in\mathcal{R}_{\nu}}w_{\nu,\beta}(\rho)
n^{f_{\nu}(\rho)}.
\end{equation}
\end{thm}

For instance,
$\langle\Tr(M^{2})\rangle_{\beta,n}$ equals 
$\dim E_{\beta,n}$.
For $\nu =(4)$ and
$\nu=(2,2)$ one gets
\begin{align*}
\langle\Tr(M^{4})\rangle_{\beta=1,n}&=\dfrac{1}{2}n^{3} 
+ \dfrac{5}{4} n^{2} + \dfrac{5}{4} n,
\qquad 
&\langle(\Tr(M^{2}))^{2}\rangle_{\beta=1,n}&=
\dfrac{1}{4}n^{4} + \dfrac{1}{2} n^{3} + \dfrac{5}{4} n^{2} + n,
\\\langle\Tr(M^{4})\rangle_{\beta=2,n}&= 2n^{3} 
+ n,
\qquad 
&\langle(\Tr(M^{2}))^{2}\rangle_{\beta=2,n}&=
n^{4} + 2 n^{2},
\\\langle\Tr(M^{4})\rangle_{\beta=4,n}&= 8n^{3} 
-10n^{2}+5n,
\qquad 
&\langle(\Tr(M^{2}))^{2}\rangle_{\beta=4,n}&=
4n^{4}-4n^{3}+5n^{2}-2n.
\end{align*}

\section{Main statements}
\label{SecStat}

\subsection{Matrix-valued free fields, isomorphisms and topological expansion}
\label{SubSecMatrixGFF}

Let $\mathcal{G}=(V,E)$ be an electrical network as in Section
\ref{ssecDynkin}.
For $\beta\in \{1,2,4\}$, $\Phi$ will be a random Gaussian function  from $V$ to $E_{\beta,n}$. The distribution of 
$\Phi$ is
\begin{equation}
\label{EqDensityMatrixGFF}
\dfrac{1}{Z_{\beta,n}^{\mathcal{G}}}\exp\Big(
-\dfrac{1}{2}\sum_{x\in V}\kappa(x)\Tr(M(x)^{2})
-\dfrac{1}{2}\sum_{\{x,y\}\in E}
C(x,y)\Tr((M(y)-M(x))^{2})
\Big)
\prod_{x\in V} dM(x).
\end{equation}
The field $\Phi$ is a matrix-valued Gaussian free field.
It can be obtained out of $\dim E_{\beta,n}$ i.i.d. copies of the scalar GFF \eqref{EqDensityGFF}, 
by considering the coefficients of the matrices.
For any $x\in V$,
$\Phi(x)/\sqrt{G(x,x)}$ is distributed as a 
G$\beta$E$(n)$ matrix.
The brackets $\langle\cdot\rangle_{\beta,n}$ will denote the expectation with respect to the law of $\Phi$ for the corresponding values of 
$(\beta,n)$.

Let $\nu=(\nu_{1},\nu_{2},\dots,\nu_{m(\nu)})$, 
where for all
$l\in\{1,2,\dots,m(\nu)\}$,
$\nu_{l}\in\mathbb{N}\setminus\{0\}$, and $\vert\nu\vert$ is even.
Let $x_{1}, x_{2},\dots , x_{\vert\nu\vert}$ be vertices in 
$V$, not necessarily distinct and $F$ a bounded measurable function
$\R^{V}\rightarrow\R$. By applying Theorem \ref{ThmIsoDynkin},
one can \textit{a priori} write an isomorphism for
\begin{displaymath}
\Big\langle \Big(\prod_{l=1}^{m(\nu)}
\Tr\Big(
\prod_{k=\nu_{1}+\dots+\nu_{l-1}+1}
^{\nu_{1}+\dots+\nu_{l}}
\Phi(x_{k})\Big)\Big)
F(\Tr(\Phi^{2})/2)\Big\rangle_{\beta,n}.
\end{displaymath}
However, if one expands the traces and the product, one gets many terms that give identical contributions, many terms with contributions that compensate, and many terms that do not contribute at all. 
Here we will be interested in the exact combinatorics that appear.
What emerges is a topological expansion, generalizing that of
Theorem \ref{ThmTopoExp1Matrix}.

Let  
$\mu_{\nu,\beta,n}^{x_{1}, x_{2},\dots , x_{\vert\nu\vert}}$ be
the following positive measure on families of $\vert\nu\vert/2$ 
nearest-neighbor paths on $\mathcal{G}$:
\begin{displaymath}
\mu_{\nu,\beta,n}^{x_{1}, x_{2},\dots , x_{\vert\nu\vert}}
(d\gamma_{1},\dots,d\gamma_{\vert\nu\vert/2})=
\sum_{\rho\in\mathcal{R}_{\nu}}
w_{\nu,\beta}(\rho) n^{f_{\nu}(\rho)}
\mu_{\textbf{p}_{\nu}(\rho)}
^{x_{1}, x_{2},\dots , x_{\vert\nu\vert}}
(d\gamma_{1},\dots,d\gamma_{\vert\nu\vert/2}).
\end{displaymath}
Next we give examples.
\begin{align*}
\mu_{\nu=(2),\beta=1,n}^{x_{1}, x_{2}} &=
\Big(\dfrac{1}{2}n^{2}+\dfrac{1}{2}n\Big)
\mu^{x_{1}, x_{2}},
\qquad &
\mu_{\nu=(1,1),\beta=1,n}^{x_{1}, x_{2}} =
n\mu^{x_{1}, x_{2}},
\\
\mu_{\nu=(2),\beta=2,n}^{x_{1}, x_{2}} &=
n^{2}\mu^{x_{1}, x_{2}},
\qquad &
\mu_{\nu=(1,1),\beta=2,n}^{x_{1}, x_{2}} =
n\mu^{x_{1}, x_{2}},
\\
\mu_{\nu=(2),\beta=4,n}^{x_{1}, x_{2}} &=
(2n^{2}-n)\mu^{x_{1}, x_{2}},
\qquad &
\mu_{\nu=(1,1),\beta=4,n}^{x_{1}, x_{2}} =
n\mu^{x_{1}, x_{2}}.
\end{align*}
The measures
$\mu_{\nu,\beta,n}^{x_{1},x_{2},x_{3},x_{4}}$
with $\vert\nu\vert = 4$ are linear combinations of
$\mu^{x_{1}, x_{2}}\otimes\mu^{x_{3}, x_{4}}$,
$\mu^{x_{1}, x_{4}}\otimes\mu^{x_{2}, x_{3}}$
and
$\mu^{x_{1}, x_{3}}\otimes\mu^{x_{2}, x_{4}}$.
We summarize the coefficients in some cases in the table below.
\begin{center}
\begin{tabular}{|c|c|c|c|}
\hline 
 & $\mu^{x_{1}, x_{2}}\otimes\mu^{x_{3}, x_{4}}$ 
 & $\mu^{x_{1}, x_{4}}\otimes\mu^{x_{2}, x_{3}}$ 
 & $\mu^{x_{1}, x_{3}}\otimes\mu^{x_{2}, x_{4}}$ \\ 
\hline 
$\mu_{\nu=(4),\beta=1,n}^{x_{1},x_{2},x_{3},x_{4}}$
& $\dfrac{1}{4}n^{3}+\dfrac{1}{2}n^{2}+\dfrac{1}{4}n$ 
& $\dfrac{1}{4}n^{3}+\dfrac{1}{2}n^{2}+\dfrac{1}{4}n$ 
& $\dfrac{1}{4}n^{2}+\dfrac{3}{4}n$ \\ 
\hline 
$\mu_{\nu=(2,2),\beta=1,n}^{x_{1},x_{2},x_{3},x_{4}}$ 
& $\dfrac{1}{4}n^{4}+\dfrac{1}{2}n^{3}+\dfrac{1}{4}n^{2}$ 
& $\dfrac{1}{2}n^{2}+\dfrac{1}{2}n$ 
& $\dfrac{1}{2}n^{2}+\dfrac{1}{2}n$ \\ 
\hline 
$\mu_{\nu=(4),\beta=2,n}^{x_{1},x_{2},x_{3},x_{4}}$ 
& $n^{3}$ & $n^{3}$ & $n$ \\ 
\hline 
$\mu_{\nu=(2,2),\beta=2,n}^{x_{1},x_{2},x_{3},x_{4}}$ 
& $n^{4}$ & $n^{2}$ & $n^{2}$ \\ 
\hline 
$\mu_{\nu=(4),\beta=4,n}^{x_{1},x_{2},x_{3},x_{4}}$ 
& $4n^{3}-4n^{2}+n$ 
& $4n^{3}-4n^{2}+n$ 
& $-2n^{2}+3n$ \\ 
\hline 
$\mu_{\nu=(2,2),\beta=4,n}^{x_{1},x_{2},x_{3},x_{4}}$ 
& $4n^{4}-4n^{3}+n^{2}$ 
& $2n^{2}-n$ & $2n^{2}-n$ \\ 
\hline 
\end{tabular}
\end{center}
In the examples of Section \ref{SubSecRibbon}, the terms in
$\mu_{\nu = (4,3,1),\beta,n}^{x_{1}, x_{2},\dots , x_{8}}$
corresponding to the pairings displayed on 
Figures \ref{FigRibbonGraphstr}, respectively \ref{FigRibbonGraph}, are measures of the form $\mu^{x_{1},x_{3}}\otimes\mu^{x_{2},x_{4}}
\otimes\mu^{x_{5},x_{8}}\otimes\mu^{x_{6},x_{7}}$ with prefactors $w_{\nu = (4,3,1),\beta}(\rho)n^{3}$, respectively
$w_{\nu = (4,3,1),\beta}(\rho)n^{2}$.

\begin{thm}
\label{ThmTopoExpField}
For $\beta\in\{1,2\}$
and $F$ a bounded measurable function 
$\mathbb{R}^{V} \rightarrow \mathbb{R}$, 
one has the following equality:
\begin{multline*}
\Big\langle \Big(\prod_{l=1}^{m(\nu)}
\Tr\Big(
\prod_{k=\nu_{1}+\dots+\nu_{l-1}+1}
^{\nu_{1}+\dots+\nu_{l}}
\Phi(x_{k})\Big)\Big)
F(\Tr(\Phi^{2})/2)\Big\rangle_{\beta,n} 
\\= 
\int\displaylimits_{\gamma_{1},\dots \gamma_{\vert\nu\vert /2}}
\Big\langle
F\big(\Tr(\Phi^{2})/2+L(\gamma_{1})+\dots + L(\gamma_{\vert\nu\vert /2})\big)
\Big\rangle_{\beta,n}
\mu_{\nu,\beta,n}^{x_{1}, x_{2},\dots , x_{\vert\nu\vert}}
(d\gamma_{1},\dots ,d\gamma_{\vert\nu\vert /2})
,
\end{multline*}
where 
$\langle\cdot\rangle_{\beta,n}
\mu_{\nu,\beta,n}^{x_{1}, x_{2},\dots ,x_{\vert\nu\vert}}
(\cdot)$ is a product measure.
For $\beta=4$, 
\begin{multline*}
\Big\langle \Big(\prod_{l=1}^{m(\nu)}
\Re\Big(\Tr\Big(
\prod_{k=\nu_{1}+\dots+\nu_{l-1}+1}
^{\nu_{1}+\dots+\nu_{l}}
\Phi(x_{k})\Big)\Big)\Big)
F(\Tr(\Phi^{2})/2)\Big\rangle_{\beta=4,n} 
\\= 
\int\displaylimits_{\gamma_{1},\dots \gamma_{\vert\nu\vert /2}}
\Big\langle
F\big(\Tr(\Phi^{2})/2+L(\gamma_{1})+\dots + L(\gamma_{\vert\nu\vert /2})\big)
\Big\rangle_{\beta=4,n}
\mu_{\nu,\beta=4,n}^{x_{1}, x_{2},\dots , x_{\vert\nu\vert}}
(d\gamma_{1},\dots ,d\gamma_{\vert\nu\vert /2})
.
\end{multline*}

In particular, if
$\lambda_{1}(x)\geq\lambda_{2}(x)\geq\dots\geq\lambda_{n}(x)$ is the family of eigenvalues of $\Phi(x)$, $x\in V$, and
\begin{displaymath}
x_{1}=\dots = x_{\nu_{1}},
\qquad
x_{\nu_{1}+1} =
\dots = x_{\nu_{1}+\nu_{2}},
\qquad
\dots ,
\qquad
x_{\vert\nu\vert-\nu_{m(\nu)}+1} = \dots
= x_{\vert\nu\vert},
\end{displaymath}
then, for $\beta\in\{1,2,4\}$,
\begin{multline}
\label{EqIsoEV}
\Big\langle \prod_{l=1}^{m(\nu)}
\Big(\sum_{i=1}^{n}
\lambda_{i}(x_{\nu_{1}+\dots + \nu_{l}})^{\nu_{l}}
\Big)
F\Big(\dfrac{1}{2}\sum_{i=1}^{n}
\lambda_{i}^{2}\Big)\Big\rangle_{\beta,n} 
\\= 
\int\displaylimits_{\gamma_{1},\dots \gamma_{\vert\nu\vert /2}}
\Big\langle
F\Big(
\dfrac{1}{2}\sum_{i=1}^{n}\lambda_{i}^{2}
+L(\gamma_{1})+\dots + L(\gamma_{\vert\nu\vert /2})\Big)
\Big\rangle_{\beta,n}
\mu_{\nu,\beta,n}
^{(x_{\nu_{1}},\nu_{1}), \dots , 
(x_{\vert\nu\vert},\nu_{m(\nu)})}
(d\gamma_{1},\dots ,d\gamma_{\vert\nu\vert /2})
,
\end{multline}
where the notation $(x_{\nu_{1}+\dots + \nu_{l}},\nu_{l})$
means that $x_{\nu_{1}+\dots + \nu_{l}}$ is repeated
$\nu_{l}$ times.
\end{thm}

\begin{rem}
\label{Rem why Re}
In the quaternionic case $\beta=4$,
when considering a single GSE matrix $M$
as in Theorem \ref{ThmTopoExp1Matrix},
one does not need to take the real part 
of $\Tr(M^{\nu_{l}})$,
since $M^{\nu_{l}}$ is quaternionic Hermitian and its diagonal entries are real.
However, in Theorem \ref{ThmTopoExpField} above,
the products
\begin{displaymath}
\prod_{k=\nu_{1}+\dots+\nu_{l-1}+1}
^{\nu_{1}+\dots+\nu_{l}}
\Phi(x_{k})
\end{displaymath}
may involve several different matrices,
and thus are not always quaternionic Hermitian.
So one needs to take the real part of the traces.
\end{rem}

\medskip

Now, we consider a family of 
$\vert\nu\vert$ deterministic square matrices with complex entries of size $n\times n$:
\begin{displaymath}
A(1,2),\dots,A(\nu_{1}-1,\nu_{1}),
A(\nu_{1},1),
\end{displaymath}
\begin{displaymath}
A(\nu_{1}+1,\nu_{1}+2),\dots,
A(\nu_{1}+\nu_{2}-1,\nu_{1}+\nu_{2}),
A(\nu_{1}+\nu_{2},\nu_{1}+1),
\end{displaymath}
\begin{displaymath}
\dots,
A(\vert\nu\vert-\nu_{m(\nu)}+1,
\vert\nu\vert-\nu_{m(\nu)}+2),
\dots,A(\vert\nu\vert-1,\vert\nu\vert),
A(\vert\nu\vert,\vert\nu\vert-\nu_{m(\nu)}+1).
\end{displaymath}
Note that by convention, for each
$l\in \{1,\dots,m(\nu)\}$ such that
$\nu_{l}=1$, we have a single matrix
$A(\nu_{1}+\dots +\nu_{l},\nu_{1}+\dots +\nu_{l})$.
For $l\in \{1,\dots,m(\nu)\}$,
$\Pi_{\nu,l}(\Phi,A)$ will denote the product
\begin{multline}
\label{EqPialpha}
\Pi_{\nu,l}(\Phi,A)=
\Phi(x_{\nu_{1}+\dots +\nu_{l-1}+1})
A(\nu_{1}+\dots +\nu_{l-1}+1,
\nu_{1}+\dots +\nu_{l-1}+2)
\Phi(x_{\nu_{1}+\dots +\nu_{l-1}+2})
\\
\dots A(\nu_{1}+\dots +\nu_{l}-1,
\nu_{1}+\dots +\nu_{l})
\Phi(x_{\nu_{1}+\dots +\nu_{l}})
A(\nu_{1}+\dots +\nu_{l},
\nu_{1}+\dots +\nu_{l-1}+1).
\end{multline}
In case $\nu_{l}=1$,
$\Pi_{\nu,l}(\Phi,A)=\Phi(x_{\nu_{1}+\dots +\nu_{l}})
A(\nu_{1}+\dots +\nu_{l},\nu_{1}+\dots +\nu_{l})$.

Next, for $\beta\in\{1,2\}$, we will write an isomorphism for
\begin{displaymath}
\Big\langle \Big(\prod_{l=1}^{m(\nu)}
\Tr\Big(\Pi_{\nu,l}(\Phi,A)\Big)\Big)
F(\Tr(\Phi^{2})/2)\Big\rangle_{\beta,n}.
\end{displaymath}
It will involve the following (complex-valued) measure
$\mu_{\nu,\beta,n,A}^{x_{1}, x_{2},\dots , x_{\vert\nu\vert}}$
on $\vert\nu\vert/2$-tuples of nearest-neighbor paths on
$\mathcal{G}$:
\begin{displaymath}
\mu_{\nu,\beta,n,A}^{x_{1}, x_{2},\dots , x_{\vert\nu\vert}}
(d\gamma_{1},\dots,d\gamma_{\vert\nu\vert/2})=
\sum_{\rho\in\mathcal{R}_{\nu}}
w_{\nu,\beta}(\rho) 
\Big(\prod_{\mathtt{t}\in\mathcal{T}_{\nu}(\rho)}
\Tr(\Pi_{\mathtt{t}}(A))
\Big)
\mu_{\textbf{p}_{\nu}(\rho)}
^{x_{1}, x_{2},\dots , x_{\vert\nu\vert}}
(d\gamma_{1},\dots,d\gamma_{\vert\nu\vert/2}).
\end{displaymath}
Here, $\Pi_{\mathtt{t}}(A)$ is a product of matrices 
of the form $A(k,k')$ or $A(k',k)^{\mathsf{T}}$,
where $^{\mathsf{T}}$ denotes the transpose 
(and not the adjoint). 
For each sequence
$k\rightarrow k'$ in the trail $\mathtt{t}$ we add the factor
$A(k,k')$ to the product, and for each
sequence $k\leftarrow k'$, we add the factor
$A(k',k)^{\mathsf{T}}$, 
all by respecting cyclic order of the trail.
For instance, if 
\begin{equation}
\label{EqOneTrail}
\mathtt{t}=
(5,\rightarrow,6,\doteq,7,\leftarrow,6,\doteq,7,\rightarrow,5,\doteq,8,\rightarrow,8,\doteq),
\end{equation}
which is one of the trails on 
Figure \eqref{FigRibbonGraph}, then
\begin{displaymath}
\Pi_{\mathtt{t}}(A) =
A(5,6)A(7,6)^{\mathsf{T}}A(7,5)A(8,8).
\end{displaymath}
Note that while the product 
$\Pi_{\mathtt{t}}(A)$ depends on the particular representative of the equivalence class $\mathtt{t}$, its trace does not. Indeed, the trace is invariant by a cyclic permutation of the factors. Moreover, reversing the direction of a representative of 
$\mathtt{t}$ amounts to taking the transpose of the product, which has the same trace.

Next we give examples of measures 
$\mu_{\nu,\beta,n,A}^{x_{1}, x_{2},\dots , x_{\vert\nu\vert}}$:
\begin{eqnarray*}
\mu_{\nu=(2),\beta=1,n,A}^{x_{1}, x_{2}}&=&
\Big(
\dfrac{1}{2}\Tr(A(1,2))\Tr(A(2,1))
+\dfrac{1}{2}\Tr(A(1,2)A(2,1)^{\mathsf{T}})
\Big)
\mu^{x_{1}, x_{2}},
\\
\mu_{\nu=(1,1),\beta=1,n,A}^{x_{1}, x_{2}}&=&
\Big(
\dfrac{1}{2}\Tr(A(1,1)A(2,2))
+\dfrac{1}{2}\Tr(A(1,1) A(2,2)^{\mathsf{T}})
\Big)
\mu^{x_{1}, x_{2}},
\\
\mu_{\nu=(2),\beta=2,n,A}^{x_{1}, x_{2}}&=&
\Tr(A(1,2))\Tr(A(2,1))
\mu^{x_{1}, x_{2}},
\\
\mu_{\nu=(1,1),\beta=2,n,A}^{x_{1}, x_{2}}&=&
\Tr(A(1,1)A(2,2))
\mu^{x_{1}, x_{2}}.
\end{eqnarray*}

Note that if all of the matrices $A(k,k')$ are equal to
$I_n$, the $n \times n$ identity matrix, then
$\mu_{\nu,\beta,n,A}^{x_{1}, x_{2},\dots , x_{\vert\nu\vert}}$ is just
$\mu_{\nu,\beta,n}^{x_{1}, x_{2},\dots , x_{\vert\nu\vert}}$, because all of the traces $\Tr(\Pi_{\mathtt{t}}(A))$ equal then $n$.

For $\beta=4$, we will need a slightly different setting. We consider matrices $\widetilde{A}(k,k')$, with the same indices
$(k,k')$ as for the matrices $A(k,k')$, but instead the 
$\widetilde{A}(k,k')$'s are $n\times n$ quaternion-valued. The 
(signed) measure 
$\mu_{\nu,\beta=4,n,\widetilde{A}}
^{x_{1}, x_{2},\dots , x_{\vert\nu\vert}}$
will be defined  as follows
\begin{displaymath}
\mu_{\nu,\beta=4,n,\widetilde{A}}
^{x_{1}, x_{2},\dots , x_{\vert\nu\vert}}
(d\gamma_{1},\dots,d\gamma_{\vert\nu\vert/2})=
\sum_{\rho\in\mathcal{R}_{\nu}}
w_{\nu,\beta=4}(\rho) 
\Big(\prod_{\mathtt{t}\in\mathcal{T}_{\nu}(\rho)}
\Re(\Tr(\widetilde{\Pi}_{\mathtt{t}}(\widetilde{A})))
\Big)
\mu_{\textbf{p}_{\nu}(\rho)}
^{x_{1}, x_{2},\dots , x_{\vert\nu\vert}}
(d\gamma_{1},\dots,d\gamma_{\vert\nu\vert/2}).
\end{displaymath}
Here, $\widetilde{\Pi}_{\mathtt{t}}(\widetilde{A})$ is a product of matrices of the form $\widetilde{A}(k,k')$ or
$\widetilde{A}(k',k)^{\ast}$.
For each sequence
$k\rightarrow k'$ in the trail $\mathtt{t}$ we add the factor
$\widetilde{A}(k,k')$ to the product, and for each
sequence $k\leftarrow k'$, we add the factor
$\widetilde{A}(k',k)^{\ast}$, 
all by respecting cyclic order of the trail. Let us emphasize that for $\beta=4$, we use the quaternion adjoint
$^{\ast}$, and not the transpose
$^{\mathsf{T}}$.
For the trail \eqref{EqOneTrail}, one gets
\begin{displaymath}
\widetilde{\Pi}_{\mathtt{t}}(\widetilde{A}) =
\widetilde{A}(5,6)\widetilde{A}(7,6)^{\ast}
\widetilde{A}(7,5)\widetilde{A}(8,8).
\end{displaymath}
Next are two examples:
\begin{eqnarray*}
\mu_{\nu=(2),\beta=4,n,\widetilde{A}}^{x_{1}, x_{2}}&=&
\Big(
2\Re(\Tr(\widetilde{A}(1,2))\Tr(\widetilde{A}(2,1)))
-\Re(\Tr(\widetilde{A}(1,2)
\widetilde{A}(2,1)^{\ast}))
\Big)
\mu^{x_{1}, x_{2}},
\\
\mu_{\nu=(1,1),\beta=4,n,\widetilde{A}}^{x_{1}, x_{2}}&=&
\Big(
\dfrac{1}{2}
\Re(\Tr(\widetilde{A}(1,1)\widetilde{A}(2,2)))
+\dfrac{1}{2}
\Re(\Tr(\widetilde{A}(1,1) 
\widetilde{A}(2,2)^{\ast}))
\Big)
\mu^{x_{1}, x_{2}}.
\end{eqnarray*}

\begin{thm}
\label{ThmTopoExpA}
For $\beta\in\{1,2\}$ and
$F$ a bounded measurable function 
$\mathbb{R}^{V} \rightarrow \mathbb{R}$, one has the following equality:
\begin{multline*}
\Big\langle \Big(\prod_{l=1}^{m(\nu)}
\Tr\Big(\Pi_{\nu,l}(\Phi,A)\Big)\Big)
F(\Tr(\Phi^{2})/2)\Big\rangle_{\beta,n} 
\\= 
\int\displaylimits_{\gamma_{1},\dots \gamma_{\vert\nu\vert /2}}
\Big\langle
F\big(\Tr(\Phi^{2})/2+L(\gamma_{1})+\dots + L(\gamma_{\vert\nu\vert /2})\big)
\Big\rangle_{\beta,n}
\mu_{\nu,\beta,n,A}^{x_{1}, x_{2},\dots , x_{\vert\nu\vert}}
(d\gamma_{1},\dots ,d\gamma_{\vert\nu\vert /2})
,
\end{multline*}
where 
$\langle\cdot\rangle_{\beta,n}
\mu_{\nu,\beta,n,A}^{x_{1}, x_{2},\dots ,x_{\vert\nu\vert}}
(\cdot)$ is a product measure.
For $\beta=4$,
\begin{multline*}
\Big\langle \Big(\prod_{l=1}^{m(\nu)}
\Re\Big(\Tr\Big(\Pi_{\nu,l}(\Phi,\widetilde{A})\Big)\Big)\Big)
F(\Tr(\Phi^{2})/2)\Big\rangle_{\beta=4,n} 
\\= 
\int\displaylimits_{\gamma_{1},\dots \gamma_{\vert\nu\vert /2}}
\Big\langle
F\big(\Tr(\Phi^{2})/2+L(\gamma_{1})+\dots + L(\gamma_{\vert\nu\vert /2})\big)
\Big\rangle_{\beta,n}
\mu_{\nu,\beta=4,n,\widetilde{A}}^{x_{1}, x_{2},\dots , 
x_{\vert\nu\vert}}
(d\gamma_{1},\dots ,d\gamma_{\vert\nu\vert /2})
.
\end{multline*}
\end{thm}

Theorem \ref{ThmTopoExpA}, which contains 
Theorem \ref{ThmTopoExpField} as a special case,
will be proved in Section \ref{SecThm2}.

\subsection{Isomorphisms and topological expansion for matrix-valued fields twisted by a connection}
\label{SubSecMatrixTwist}

We will denote by $\Ub_{\beta,n}$ the group $O(n)$ if $\beta=1$, 
$U(n)$ if $\beta=2$, and $U(n,\mathbb{H})$ if $\beta=4$. 
We consider a connection on $\mathcal{G}$,
$(U(x,y))_{\{x,y\}\in E}$, with
$U(x,y)\in \Ub_{\beta,n}$ and
\begin{displaymath}
U(y,x)=U(x,y)^{\ast}=U(x,y)^{-1}.
\end{displaymath}
Let $\langle\cdot\rangle_{\beta,n}^{U}$ be the following probability measure on $(E_{\beta,n})^{V}$, defined by the density
\begin{equation}
\label{EqDensityU}
\dfrac{1}{Z^{\mathcal{G},U}_{\beta,n}}
\exp\Big(
-\dfrac{1}{2}\sum_{x\in V}\kappa(x)\Tr(M(x)^{2})
-\dfrac{1}{2}\sum_{\{x,y\}\in E}
C(x,y)\Tr((M(y)-U(y,x)M(x)U(x,y))^{2})
\Big).
\end{equation}
Note that if $\{x,y\}\in E$ , then
\begin{displaymath}
\Tr((M(x)-U(x,y)M(y)U(y,x))^{2})
=
\Tr((M(y)-U(y,x)M(x)U(x,y))^{2}).
\end{displaymath}
We will denote by $\Phi$ the field under the measure
$\langle\cdot\rangle_{\beta,n}^{U}$. 
It is the matrix-valued 
Gaussian free field twisted by the connection $U$.
If $(\mathfrak{U}(x))_{x\in V}$ is a gauge transformation, then
$(\mathfrak{U}(x)^{-1}\Phi(x)\mathfrak{U}(x))_{x\in V}$ is the
field associated to the connection
$(\mathfrak{U}(x)^{-1}U(x,y)\mathfrak{U}(y))_{\{x,y\}\in E}$.
If the connection $U$ is trivial, then for all
$x\in V$, $\Phi(x)/\sqrt{G(x,x)}$ is distributed as a 
G$\beta$E$(n)$ matrix. 
This is not necessarily the case if the connection is non-trivial.

As in Section \ref{SubSecMatrixGFF},
we take 
$\nu=(\nu_{1},\nu_{2},\dots,\nu_{m(\nu)})$, where for all
$l\in\{1,2,\dots,m(\nu)\}$,
$\nu_{l}\in\mathbb{N}\setminus\{0\}$, and $\vert\nu\vert$ is even.
Let $x_{1}, x_{2},\dots , x_{\vert\nu\vert}$ be vertices in 
$V$, not necessarily distinct. As in Section \ref{SubSecMatrixGFF}, the $A(k,k')$'s respectively
$\widetilde{A}(k,k')$'s are 
families of $\vert\nu\vert$ square matrices with complex, respectively quaternionic entries of size $n\times n$.
The products $\Pi_{\nu,l}(\Phi,A)$ and 
$\Pi_{\nu,l}(\Phi,\widetilde{A})$ are defined as in
\eqref{EqPialpha}.

Next we will write an isomorphism for
\begin{displaymath}
\Big\langle \Big(\prod_{l=1}^{m(\nu)}
\Tr\Big(\Pi_{\nu,l}(\Phi,A)\Big)\Big)
F(\Tr(\Phi^{2})/2)\Big\rangle_{\beta,n}^{U},
\qquad \beta\in\{1,2\},
\end{displaymath}
and for
\begin{displaymath}
\Big\langle \Big(\prod_{l=1}^{m(\nu)}
\Re\Big(\Tr\Big(\Pi_{\nu,l}(\Phi,\widetilde{A})\Big)\Big)\Big)
F(\Tr(\Phi^{2})/2)\Big\rangle_{\beta,n}^{U},
\qquad \beta=4.
\end{displaymath}

For $\beta=1$,
we introduce the following (complex) measure
$\mu_{\nu,\beta=1,n,A,U}
^{x_{1}, x_{2},\dots , x_{\vert\nu\vert}}$
on $\vert\nu\vert/2$-tuples of nearest-neighbor paths on
$\mathcal{G}$:
\begin{multline*}
\mu_{\nu,\beta=1,n,A,U}^{x_{1}, x_{2},\dots , x_{\vert\nu\vert}}
(d\gamma_{1},\dots,d\gamma_{\vert\nu\vert/2})=\\
\sum_{\rho\in\mathcal{R}_{\nu}}
w_{\nu,\beta=1}(\rho) 
\Big(\prod_{\mathtt{t}\in\mathcal{T}_{\nu}(\rho)}
\Tr(
\Pi_{\mathtt{t}}(U,A)(\gamma_{1},\dots,\gamma_{\vert\nu\vert /2})
)
\Big)
\mu_{\textbf{p}_{\nu}(\rho)}
^{x_{1}, x_{2},\dots , x_{\vert\nu\vert}}
(d\gamma_{1},\dots,d\gamma_{\vert\nu\vert/2}).
\end{multline*}
Here, 
$\Pi_{\mathtt{t}}(U,A)(\gamma_{1},\dots,\gamma_{\vert\nu\vert /2})$ 
is a product of matrices 
of the form $A(k,k')$ or $A(k',k)^{\mathsf{T}}$, 
and $\mathsf{hol}^{U}(\gamma)$ or
$\mathsf{hol}^{U}(\gamma)^{\ast}$, $\gamma$ being one of the paths. 
For each sequence
$k\rightarrow k'$ in the trail $\mathtt{t}$ we add the factor
$A(k,k')$ to the product, and for each
sequence $k\leftarrow k'$, we add the factor
$A(k',k)^{\mathsf{T}}$, as in the construction of
$\mu_{\nu,\beta,n,A}^{x_{1}, x_{2},\dots , x_{\vert\nu\vert}}$.
Moreover, for each sequence $k\doteq k'$ with $k<k'$, a measure
$\mu^{x_{k},x_{k'}}(d\gamma_{i})$ is present, and we add to the product the factor $\mathsf{hol}^{U}(\gamma_{i})$.
For each sequence $k\doteq k'$ with $k>k'$, a measure
$\mu^{x_{k'},x_{k}}(d\gamma_{i})$ is present, and we add to the product the factor $\mathsf{hol}^{U}(\gamma_{i})^{\ast}
=\mathsf{hol}^{U}(\gamma_{i})^{\mathsf{T}}$.
In the product, the factors respect the cyclic order on the trail.

For $\beta=2$, we will need oriented trails. 
The (complex) measure 
$\mu_{\nu,\beta=2,n,A,U}^{x_{1}, x_{2},\dots , x_{\vert\nu\vert}}$
on  $\vert\nu\vert/2$-tuples of nearest-neighbor paths on
$\mathcal{G}$ is as follows:
\begin{multline*}
\mu_{\nu,\beta=2,n,A,U}^{x_{1}, x_{2},\dots , x_{\vert\nu\vert}}
(d\gamma_{1},\dots,d\gamma_{\vert\nu\vert/2})=\\
\sum_{\rho\in\mathcal{R}_{\nu}}
w_{\nu,\beta=2}(\rho) 
\Big(\prod_{\scriptsize{\overrightarrow{\mathtt{t}}
\in\overrightarrow{\mathcal{T}}_{\nu}(\rho)}}
\Tr(
\Pi_{\scriptsize{\overrightarrow{\mathtt{t}}}}
(U,A)(\gamma_{1},\dots,\gamma_{\vert\nu\vert /2})
)
\Big)
\mu_{\textbf{p}_{\nu}(\rho)}
^{x_{1}, x_{2},\dots , x_{\vert\nu\vert}}
(d\gamma_{1},\dots,d\gamma_{\vert\nu\vert/2}).
\end{multline*}
Here, $\Pi_{\scriptsize{\overrightarrow{\mathtt{t}}}}
(U,A)(\gamma_{1},\dots,\gamma_{\vert\nu\vert /2})$ 
is a product of matrices 
of the form $A(k,k')$  and $\mathsf{hol}^{U}(\gamma)$ or
$\mathsf{hol}^{U}(\gamma)^{\ast}$, $\gamma$ being one of the paths. 
For each sequence
$k\rightarrow k'$ in the oriented trail 
$\overrightarrow{\mathtt{t}}$ we add the factor
$A(k,k')$ to the product.
Moreover, for each sequence $k\doteq k'$ with $k<k'$, a measure
$\mu^{x_{k},x_{k'}}(d\gamma_{i})$ is present, and we add to the product the factor $\mathsf{hol}^{U}(\gamma_{i})$.
For each sequence $k\doteq k'$ with $k>k'$, a measure
$\mu^{x_{k'},x_{k}}(d\gamma_{i})$ is present, and we add to the product the factor $\mathsf{hol}^{U}(\gamma_{i})^{\ast}$.
In the product, the factors respect the cyclic order on the trail.
The reason we use oriented trails is that the reversal of the orientation of the trail changes the trace, as complex adjoints of holonomies appear.

For $\beta=4$, we return to unoriented trails.
We introduce the (signed) measure
$\mu_{\nu,\beta=4,n,\widetilde{A},U}^{x_{1}, x_{2},\dots, x_{\vert\nu\vert}}$
on $\vert\nu\vert/2$-tuples of nearest-neighbor paths in
$\mathcal{G}$, constructed as follows:
\begin{multline*}
\mu_{\nu,\beta=4,n,\widetilde{A},U}
^{x_{1}, x_{2},\dots, x_{\vert\nu\vert}}
(d\gamma_{1},\dots,d\gamma_{\vert\nu\vert/2})=\\
\sum_{\rho\in\mathcal{R}_{\nu}}
w_{\nu,\beta=4}(\rho) 
\Big(\prod_{\mathtt{t}\in\mathcal{T}_{\nu}(\rho)}
\Re(
\Tr(
\widetilde{\Pi}_{\mathtt{t}}(U,\widetilde{A})
(\gamma_{1},\dots,\gamma_{\vert\nu\vert /2})
))
\Big)
\mu_{\textbf{p}_{\nu}(\rho)}
^{x_{1}, x_{2},\dots , x_{\vert\nu\vert}}
(d\gamma_{1},\dots,d\gamma_{\vert\nu\vert/2}).
\end{multline*}
Here, $\widetilde{\Pi}_{\mathtt{t}}(U,\widetilde{A})
(\gamma_{1},\dots,\gamma_{\vert\nu\vert /2})$ 
is a product of matrices 
of the form $\widetilde{A}(k,k')$ or 
$\widetilde{A}(k',k)^{\ast}$, 
and $\mathsf{hol}^{U}(\gamma)$ or
$\mathsf{hol}^{U}(\gamma)^{\ast}$, $\gamma$ being one of the paths. 
For each sequence
$k\rightarrow k'$ in the trail $\mathtt{t}$ we add the factor
$\widetilde{A}(k,k')$ to the product, and for each
sequence $k\leftarrow k'$, we add the factor
$\widetilde{A}(k',k)^{\ast}$.
Moreover, for each sequence $k\doteq k'$ with $k<k'$, a measure
$\mu^{x_{k},x_{k'}}(d\gamma_{i})$ is present, and we add to the product the factor $\mathsf{hol}^{U}(\gamma_{i})$.
For each sequence $k\doteq k'$ with $k>k'$, a measure
$\mu^{x_{k'},x_{k}}(d\gamma_{i})$ is present, and we add to the product the factor $\mathsf{hol}^{U}(\gamma_{i})^{\ast}$.
In the product, the factors respect the cyclic order on the trail.

Next are some examples of measures 
$\mu_{\nu,\beta,n,A,U}
^{x_{1},x_{2},\dots,x_{\vert\nu\vert}}$ and
$\mu_{\nu,\beta=4,n,\widetilde{A},U}
^{x_{1},x_{2},\dots,x_{\vert\nu\vert}}$:
\begin{eqnarray*}
\mu_{\nu=(2),\beta=1,n,A,U}^{x_{1}, x_{2}}&=&
\Big(
\dfrac{1}{2}\Tr(A(1,2)\mathsf{hol}^{U}(\gamma)^{\ast})
\times\Tr(A(2,1)\mathsf{hol}^{U}(\gamma))
\\&&
+\dfrac{1}{2}\Tr(A(1,2)\mathsf{hol}^{U}(\gamma)^{\ast}
A(2,1)^{\mathsf{T}}\mathsf{hol}^{U}(\gamma))
\Big)
\mu^{x_{1}, x_{2}}(d\gamma),
\\
\mu_{\nu=(1,1),\beta=1,n,A,U}^{x_{1}, x_{2}}&=&
\Big(
\dfrac{1}{2}\Tr(A(1,1)\mathsf{hol}^{U}(\gamma)A(2,2)
\mathsf{hol}^{U}(\gamma)^{\ast})
\\&&
+\dfrac{1}{2}\Tr(A(1,1)\mathsf{hol}^{U}(\gamma)
A(2,2)^{\mathsf{T}}\mathsf{hol}^{U}(\gamma)^{\ast})
\Big)
\mu^{x_{1}, x_{2}}(d\gamma),
\end{eqnarray*}
\begin{eqnarray*}
\mu_{\nu=(2),\beta=2,n,A,U}^{x_{1}, x_{2}}&=&
\Tr(A(1,2)\mathsf{hol}^{U}(\gamma)^{\ast})
\times\Tr(A(2,1)\mathsf{hol}^{U}(\gamma))
\mu^{x_{1}, x_{2}}(d\gamma),
\\
\mu_{\nu=(1,1),\beta=2,n,A,U}^{x_{1}, x_{2}}&=&
\Tr(A(1,1)\mathsf{hol}^{U}(\gamma)
A(2,2)\mathsf{hol}^{U}(\gamma)^{\ast})
\mu^{x_{1}, x_{2}}(d\gamma),
\end{eqnarray*}
\begin{eqnarray*}
\mu_{\nu=(2),\beta=4,n,\widetilde{A},U}^{x_{1}, x_{2}}&=&
\Big(
2\Re(\Tr(\widetilde{A}(1,2)\mathsf{hol}^{U}(\gamma)^{\ast}))
 \times\Re(\Tr(\widetilde{A}(2,1)\mathsf{hol}^{U}(\gamma)))
\\&&
-\Re(\Tr(\widetilde{A}(1,2)\mathsf{hol}^{U}(\gamma)^{\ast}
\widetilde{A}(2,1)^{\ast}\mathsf{hol}^{U}(\gamma)))
\Big)
\mu^{x_{1}, x_{2}}(d\gamma),
\\
\mu_{\nu=(1,1),\beta=4,n,\widetilde{A},U}^{x_{1}, x_{2}}&=&
\Big(
\dfrac{1}{2}\Re(\Tr(\widetilde{A}(1,1)\mathsf{hol}^{U}(\gamma)
\widetilde{A}(2,2)\mathsf{hol}^{U}(\gamma)^{\ast}))
\\&&
+\dfrac{1}{2}\Re(\Tr(\widetilde{A}(1,1)\mathsf{hol}^{U}(\gamma)
\widetilde{A}(2,2)^{\ast}\mathsf{hol}^{U}(\gamma)^{\ast}))
\Big)
\mu^{x_{1}, x_{2}}(d\gamma).
\end{eqnarray*}

\begin{thm}
\label{ThmTopoExpAU}
For $\beta\in\{1,2\}$ and
$F$ a bounded measurable function 
$\mathbb{R}^{V} \rightarrow \mathbb{R}$, 
one has the following equality:
\begin{multline}
\label{EqIsoMatrixAU}
\Big\langle \Big(\prod_{l=1}^{m(\nu)}
\Tr\Big(\Pi_{\nu,l}(\Phi,A)\Big)\Big)
F(\Tr(\Phi^{2})/2)\Big\rangle_{\beta,n}^{U}
\\= 
\int\displaylimits_{\gamma_{1},\dots \gamma_{\vert\nu\vert /2}}
\Big\langle
F\big(\Tr(\Phi^{2})/2
+L(\gamma_{1})+\dots + L(\gamma_{\vert\nu\vert /2})\big)
\Big\rangle_{\beta,n}^{U}
\mu_{\nu,\beta,n,A,U}^{x_{1}, x_{2},\dots , x_{\vert\nu\vert}}
(d\gamma_{1},\dots ,d\gamma_{\vert\nu\vert /2})
,
\end{multline}
where 
$\langle\cdot\rangle_{\beta,n}^{U}
\mu_{\nu,\beta,n,A,U}^{x_{1}, x_{2},\dots ,x_{\vert\nu\vert}}
(\cdot)$ is a product measure.
For $\beta=4$,
\begin{multline}
\label{EqIsoMatrixAU4}
\Big\langle \Big(\prod_{l=1}^{m(\nu)}
\Re\Big(\Tr\Big(\Pi_{\nu,l}(\Phi,\widetilde{A})\Big)\Big)\Big)
F(\Tr(\Phi^{2})/2)\Big\rangle_{\beta=4,n}^{U}
\\= 
\int\displaylimits_{\gamma_{1},\dots \gamma_{\vert\nu\vert /2}}
\Big\langle
F\big(\Tr(\Phi^{2})/2
+L(\gamma_{1})+\dots + L(\gamma_{\vert\nu\vert /2})\big)
\Big\rangle_{\beta=4,n}^{U}
\mu_{\nu,\beta=4,n,\widetilde{A},U}
^{x_{1}, x_{2},\dots , x_{\vert\nu\vert}}
(d\gamma_{1},\dots ,d\gamma_{\vert\nu\vert /2}).
\end{multline}
\end{thm}

Theorem \ref{ThmTopoExpAU} will be proved in Section \ref{SecThm3}.

\begin{rem}
Note that in the measures
$\mu_{\nu,\beta,n,A,U}^{x_{1}, x_{2},\dots,x_{\vert\nu\vert}}
(d\gamma_{1},\dots ,d\gamma_{\vert\nu\vert /2})$
for $\beta\in\{1,2\}$,  and in
$\mu_{\nu,\beta=4,n,\widetilde{A},U}
^{x_{1}, x_{2},\dots,x_{\vert\nu\vert}}
(d\gamma_{1},\dots ,d\gamma_{\vert\nu\vert /2})$,
the holonomy along each path $\gamma_{i}$ appears twice.
\end{rem}

\begin{rem}
If one does not have 
$x_{\nu_{1}+\dots + \nu_{l-1} + 1}
= x_{\nu_{1}+\dots + \nu_{l-1} + 2}
= \dots =  x_{\nu_{1}+\dots + \nu_{l}}$,
then from a geometrical viewpoint it is not very natural
to consider just the trace
\begin{displaymath}
\Tr\Big(
\prod_{k=\nu_{1}+\dots + \nu_{l-1} + 1}^{\nu_{1}+\dots + \nu_{l}}
\Phi(x_{k})
\Big) .
\end{displaymath}
This is because then the $\Phi(x_{k})$ live on fibers
above different points of the base.
So one needs a way to compare fibers above different points.
For this one can intertwine the matrices
$A(k,k')$ and $\widetilde{A}(k,k')$, 
with $A(k,k')$ real symmetric $(\beta = 1)$,
respectively complex Hermitian $(\beta = 2)$,
and $\widetilde{A}(k,k')$
quaternionic Hermitian $(\beta = 4)$.
However, it turns out that
the identities of Theorem \ref{ThmTopoExpAU}
are the same for $A(k,k')$ and $\widetilde{A}(k,k')$
being more general.
\end{rem}

\begin{rem}
Consider the particular case when the connection $U$ is trivial, 
i.e. for any closed path (loop) $\gamma$, 
$\mathsf{hol}^{U}(\gamma)=I_n$. 
There is a gauge transformation
$\mathfrak{U}:V\rightarrow\Ub_{\beta,n}$
such that
\begin{displaymath}
\forall x,y\in V
\text{ such that }
\{x,y\}\in E,~~
\mathfrak{U}(x)^{-1}\mathfrak{U}(y) = U(x,y).
\end{displaymath}
Then for any
$x,y,\in V$ and $\gamma$ nearest-neighbor path from $x$ to $y$,
\begin{displaymath}
\mathsf{hol}^{U}(\gamma)=
\mathfrak{U}(x)^{-1}\mathfrak{U}(y).
\end{displaymath}
The field $\Phi$ under $\langle\cdot\rangle^{U}_{\beta,n}$ 
has the same law as 
$(\mathfrak{U}(x)^{-1}\Phi(x)\mathfrak{U}(x))_{x\in V}$ under
$\langle\cdot\rangle_{\beta,n}$.
So, in the particular case of a trivial connection,
Theorem \ref{ThmTopoExpAU} follows directly from
Theorem \ref{ThmTopoExpA}. 
\end{rem}

\medskip

A special case of particular interest in Theorem \ref{ThmTopoExpAU}
is when all the matrices $A(k,k')$ and 
$\widetilde{A}(k,k')$ equal $I_n$ and
\begin{displaymath}
x_{1}=\dots = x_{\nu_{1}},
\qquad
x_{\nu_{1}+1} =
\dots = x_{\nu_{1}+\nu_{2}},
\qquad
\dots ,
\qquad
x_{\vert\nu\vert-\nu_{m(\nu)}+1} = \dots
= x_{\vert\nu\vert}.
\end{displaymath}
Then, on the left-hand side of
\eqref{EqIsoMatrixAU}, we have
\begin{displaymath}
\forall l\in \{1,\dots,m(\nu)\},
\Tr(\Pi_{\nu,l}(\Phi,I_{n}))
=\sum_{i=1}^{n}\lambda_{i}(
x_{\nu_{1}+\dots +\nu_{l}})^{\nu_{l}},
\end{displaymath}
where
$\lambda_{1}(x)\geq\lambda_{2}(x)\geq\dots \geq\lambda_{n}(x)$ 
is the family of eigenvalues of $\Phi$.
On the right-hand side of
\eqref{EqIsoMatrixAU} appears a product of Wilson loops.
Indeed, each
$\Pi_{\mathtt{t}}(U,I_{n})(\gamma_{1},\dots,\gamma_{\vert\nu\vert /2})$,
$\Pi_{\scriptsize{\overrightarrow{\mathtt{t}}}}(U,I_{n})(\gamma_{1},\dots,\gamma_{\vert\nu\vert /2})$
and
$\widetilde{\Pi}_{\mathtt{t}}(U,I_{n})
(\gamma_{1},\dots,\gamma_{\vert\nu\vert /2})$ 
is then a holonomy along a loop formed by concatenating some of the paths $\gamma_{k}$. In the example of Figure
\ref{FigRibbonGraphstr}, the holonomies are
\begin{displaymath}
\mathsf{hol}^{U}(\gamma_{1})\mathsf{hol}^{U}(\gamma_{2})
\mathsf{hol}^{U}(\overleftarrow{\gamma_{1}})
\mathsf{hol}^{U}(\overleftarrow{\gamma_{2}}),
\qquad
\mathsf{hol}^{U}(\gamma_{3})\mathsf{hol}^{U}(\gamma_{4})
\mathsf{hol}^{U}(\overleftarrow{\gamma_{4}}),
\qquad \text{and}
\qquad
\mathsf{hol}^{U}(\overleftarrow{\gamma_{3}}).
\end{displaymath}
where $\gamma_{1}$ and $\gamma_{2}$ are paths from 
$x_{4}$ to $x_{4}$, $\gamma_{3}$ is a path from
from $x_{7}$ to $x_{7}$ and $\gamma_{4}$ a path from
$x_{7}$ to $x_{8}$. In the example of Figure
\ref{FigRibbonGraph}, we get
\begin{displaymath}
\mathsf{hol}^{U}(\gamma_{1})
\mathsf{hol}^{U}(\overleftarrow{\gamma_{2}})
\mathsf{hol}^{U}(\overleftarrow{\gamma_{1}})
\mathsf{hol}^{U}(\overleftarrow{\gamma_{2}}),
\qquad
\text{and}
\qquad
\mathsf{hol}^{U}(\gamma_{3})\mathsf{hol}^{U}(\gamma_{3})\mathsf{hol}^{U}(\gamma_{4})
\mathsf{hol}^{U}(\overleftarrow{\gamma_{4}}).
\end{displaymath}
Note that since the joint distribution of all eigenvalues above all vertices is invariant under gauge transformations, 
it is natural that expectations with respect to only these eigenvalues involve only Wilson loops;
see Section \ref{SubSecConx} and 
\cite{Giles81Reconstruct,Sengupta94GaugeInvariant,TLevy04WilsonLoops}. 
We summarize this paragraph in the following corollary.

\begin{cor}
\label{CorWilsonLoops}
For $\beta\in\{1,2,4\}$ and $U$ a non-trivial connection, 
in the isomorphism for
\begin{displaymath}
\Big\langle \prod_{l=1}^{m(\nu)}
\Big(\sum_{i=1}^{n}
\lambda_{i}(x_{\nu_{1}+\dots + \nu_{l}})^{\nu_{l}}
\Big)
F\Big(\dfrac{1}{2}\sum_{i=1}^{n}
\lambda_{i}^{2}\Big)\Big\rangle_{\beta,n}^{U},
\end{displaymath}
compared to the case of a trivial connection
\eqref{EqIsoEV}, each power of $n$ in the topological expansion, that is $n^{f_{\nu}(\rho)}$, is replaced by a product of
$f_{\nu}(\rho)$ Wilson loops corresponding to holonomies along random walk loops, 
one for each boundary cycle of the ribbon in the ribbon pairing $\rho$.
\end{cor}

\begin{rem}
\label{RemTraceGFF}
The density \eqref{EqDensityU} can be factorized
by using the orthogonal decomposition
\begin{displaymath}
M=\frac{1}{n}\Tr(M)I_{n}+\Big(M-\frac{1}{n}\Tr(M)I_{n}\Big).
\end{displaymath}
We have that
\begin{displaymath}
\Tr(M(x)^{2})
=
\dfrac{1}{n}\Tr(M(x))^{2}
+
\Tr\Big(\Big(M(x)-\frac{1}{n}\Tr(M(x))I_{n}\Big)^{2}\Big),
\end{displaymath}
and
\begin{multline*}
\Tr((M(y)-U(y,x)M(x)U(x,y))^{2})
=
\dfrac{1}{n}(\Tr(M(y))-\Tr(M(x)))^{2}
\\+
\Tr\Big(\Big(M(y)-\dfrac{1}{n}\Tr(M(y))I_{n}
-U(y,x)\Big(M(x)-\dfrac{1}{n}\Tr(M(x))I_{n}\Big)U(x,y)\Big)^{2}\Big).
\end{multline*}
So, under $\langle\cdot\rangle_{\beta,n}^{U}$,
for $\beta\in\{1,2,4\}$,
the matrix-valued field
$\Phi-\frac{1}{n}\Tr(\Phi)I_{n}$
and the scalar field $\Tr(\Phi)$
are independent.
Moreover, the field
\begin{displaymath}
\Big(\dfrac{1}{\sqrt{n}}\Tr(\Phi(x))\Big)_{x\in V}=
\Big(\dfrac{1}{\sqrt{n}}\sum_{i}^{n}\lambda_{i}(x)\Big)_{x\in V}
\end{displaymath}
is distributed as a scalar GFF \eqref{EqDensityGFF},
and in particular its law is the same whatever the connection $U$.
The latter point can also be seen through the covariance
structure of $\Tr(\Phi)$.
In the expression of two-point correlations
\begin{displaymath}
\langle \Tr(\Phi(x))\Tr(\Phi(y))\rangle_{\beta,n}^{U}
\end{displaymath}
appear only ribbon graphs with one boundary component
$(f_{\nu=(1,1)}(\rho)=1)$. 
Thus, the two-point correlations are expressed with
\begin{displaymath}
\Tr(\mathsf{hol}^{U}(\gamma)\mathsf{hol}^{U}(\gamma)^{\ast})
\mu^{x,y}(d\gamma)=n\mu^{x,y}(d\gamma).
\end{displaymath}
Hence,
\begin{displaymath}
\langle \Tr(\Phi(x))\Tr(\Phi(y))\rangle_{\beta,n}^{U}
= n G(x,y).
\end{displaymath}
\end{rem}

\section{Proofs}
\label{SecProofs}

\subsection{Proof of Theorem \ref{ThmTopoExpA}}
\label{SecThm2}

We give the proof of Theorem \ref{ThmTopoExpA}, which contains 
Theorem \ref{ThmTopoExpField} as a special case.
A possible approach would be to expand the product of traces
on the left-hand side of the isomorphism identity, apply Theorem
\ref{ThmIsoDynkin}, and then recombine the terms to get the right-hand side of the isomorphism identity. This works well for
$\beta\in\{1,2\}$ (see \cite{Zvonkin97MatrixIntegrals} for one matrix integrals), but for $\beta=4$ this approach is less tractable because of the non-commutativity of quaternions. 
Instead, we will rely on an induction over the numbers of edges
$\vert\nu\vert/2$, 
similar to that in \cite{BrycPierce09GSE}.

We use the notations of Section \ref{SubSecMatrixGFF}.
For $\beta\in\{1,2,4\}$, let $(M^{(i)})_{i\geq 1}$ be an i.i.d. sequence of G$\beta$E(n) matrices with distribution
\eqref{EqGbetaEmat}. Given
$\textbf{p}=\{\{a_{1},b_{1}\},\dots,
\{a_{\vert\nu\vert/2},b_{\vert\nu\vert/2}\}\}$ a partition of
$\{1,\dots, \vert\nu\vert\}$ in pairs
and $l\in \{1,\dots,m(\nu)\}$,
$\Pi_{\nu,l,\textbf{p}}(M,A)$ will denote the product of matrices
\begin{multline*}
\Pi_{\nu,l,\textbf{p}}(M,A)=
M^{(\I_{\nu_{1}+\dots +\nu_{l-1}+1})}
A(\nu_{1}+\dots +\nu_{l-1}+1,
\nu_{1}+\dots +\nu_{l-1}+2)
M^{(\I_{\nu_{1}+\dots +\nu_{l-1}+2})}
\\
\dots A(\nu_{1}+\dots +\nu_{l}-1,
\nu_{1}+\dots +\nu_{l})
M^{(\I_{\nu_{1}+\dots +\nu_{l}})}
A(\nu_{1}+\dots +\nu_{l},
\nu_{1}+\dots +\nu_{l-1}+1),
\end{multline*}
where,
for $k\in \{1,\dots ,\vert\nu\vert\}$,
we denote by $\I_{k}$ the unique index 
$i\in\{1,\dots ,\vert\nu\vert/2\}$ such that
$k\in \{a_{i},b_{i}\}$.
Note that how exactly the pairs $\{a_{i},b_{i}\}$ 
in the partition $\textbf{p}$ are ordered will not be important.
The product $\Pi_{\nu,l,\textbf{p}}(M,\widetilde{A})$ is defined similarly, with matrices
$\widetilde{A}(k,k')$ instead of
$A(k,k')$. Note that in the products
\begin{displaymath}
\prod_{l=1}^{m(\nu)}
\Tr\Big(\Pi_{\nu,l,\textbf{p}}(M,A)\Big)
\qquad
\text{and}
\qquad
\prod_{l=1}^{m(\nu)}
\Re\Big(\Tr\Big(\Pi_{\nu,l,\textbf{p}}(M,\widetilde{A})
\Big)\Big),
\end{displaymath}
each $M^{(i)}$ for $i\in\{1,\dots,\vert\nu\vert/2\}$ appears exactly twice, at positions corresponding to a pair in the partition $\textbf{p}$.

\begin{lemma}
\label{LemMi}
Then, for $\beta\in\{1,2\}$ and
$F$ a bounded measurable function 
$\mathbb{R}^{V} \rightarrow \mathbb{R}$, one has the following equality:
\begin{multline*}
\Big\langle \Big(\prod_{l=1}^{m(\nu)}
\Tr\Big(\Pi_{\nu,l}(\Phi,A)\Big)\Big)
F(\Tr(\Phi^{2})/2)\Big\rangle_{\beta,n} 
\\= 
\sum_{\substack{\textbf{p} \text{ partition}
\\\text{of } \{1,\dots, \vert\nu\vert\}
\\\text{in pairs }
}}
\int\displaylimits_{\gamma_{1},\dots \gamma_{\vert\nu\vert /2}}
\Big\langle \prod_{l=1}^{m(\nu)}
\Tr\Big(\Pi_{\nu,l,\textbf{p}}(M,A)\Big)
\Big\rangle_{\beta,n}
\\
\times
\Big\langle
F\big(\Tr(\Phi^{2})/2+L(\gamma_{1})+\dots + L(\gamma_{\vert\nu\vert /2})\big)
\Big\rangle_{\beta,n}
\mu^{x_{1},x_{2},\dots, x_{\vert\nu\vert}}_{\textbf{p}}
(d\gamma_{1},\dots,d\gamma_{\vert\nu\vert /2}).
\end{multline*}
For $\beta=4$,
\begin{multline*}
\Big\langle \Big(\prod_{l=1}^{m(\nu)}
\Re\Big(\Tr\Big(\Pi_{\nu,l}(\Phi,\widetilde{A})\Big)\Big)\Big)
F(\Tr(\Phi^{2})/2)\Big\rangle_{\beta=4,n} 
\\= 
\sum_{\substack{\textbf{p} \text{ partition}
\\\text{of } \{1,\dots, \vert\nu\vert\}
\\\text{in pairs }
}}
\int\displaylimits_{\gamma_{1},\dots \gamma_{\vert\nu\vert /2}}
\Big\langle \prod_{l=1}^{m(\nu)}
\Re\Big(\Tr\Big(\Pi_{\nu,l,\textbf{p}}(M,\widetilde{A})
\Big)\Big)
\Big\rangle_{\beta=4,n}
\\
\times
\Big\langle
F\big(\Tr(\Phi^{2})/2+L(\gamma_{1})+\dots + L(\gamma_{\vert\nu\vert /2})\big)
\Big\rangle_{\beta=4,n}
\mu^{x_{1},x_{2},\dots, x_{\vert\nu\vert}}_{\textbf{p}}
(d\gamma_{1},\dots,d\gamma_{\vert\nu\vert /2}).
\end{multline*}
\end{lemma}

\begin{proof}
On the left-hand side one can expand the product of traces
\begin{displaymath}
\prod_{l=1}^{m(\nu)}
\Tr\Big(\Pi_{\nu,l}(\Phi,A)\Big)
\qquad
\text{and}
\qquad
\prod_{l=1}^{m(\nu)}
\Re\Big(\Tr\Big(\Pi_{\nu,l}(\Phi,\widetilde{A})\Big)\Big).
\end{displaymath}
 Then, one applies to each term in the sum Lemma
 \ref{LemIsoKLrew} (in the particular case
 $U\equiv I_{n}$), so as to make appear the entries of
 G$\beta$E$(n)$ matrices $M^{(i)}$. 
 Then one factorizes these entries into
 \begin{displaymath}
\prod_{l=1}^{m(\nu)}
\Tr\Big(\Pi_{\nu,l,\textbf{p}}(M,A)\Big)
\qquad
\text{and}
\qquad
\prod_{l=1}^{m(\nu)}
\Re\Big(\Tr\Big(\Pi_{\nu,l,\textbf{p}}(M,\widetilde{A})
\Big)\Big).
\qedhere
\end{displaymath}
\end{proof}

It remains to compute the moments
\begin{displaymath}
\langle\prod_{l=1}^{m(\nu)}
\Tr\Big(\Pi_{\nu,l,\textbf{p}}(M,A)\Big)
\Big\rangle_{\beta,n}
\qquad
\text{and}
\qquad
\Big\langle \prod_{l=1}^{m(\nu)}
\Re\Big(\Tr\Big(\Pi_{\nu,l,\textbf{p}}(M,\widetilde{A})
\Big)\Big)
\Big\rangle_{\beta=4,n}.
\end{displaymath}
for $\vert\nu\vert/2$ i.i.d. 
G$\beta$E$(n)$ matrices. For $\beta=4$ we will use expressions of moments of quaternionic Gaussian r.v. that appeared in
\cite{BrycPierce09GSE}.

\begin{lemma}[Bryc-Pierce \cite{BrycPierce09GSE}]
\label{LemQuaternionWick}
Let $\xi\in\mathbb{H}$ be a quaternionic Gaussian r.v.
with distribution
\begin{equation}
\label{EqQuatGauss}
\dfrac{1}{(2\pi)^{2}}e^{-\frac{1}{2}\vert q\vert^{2}} 
dq_{r} dq_{i} dq_{j} dq_{k}.
\end{equation}
Then, for any $q_{1},q_{2}\in\mathbb{H}$,
\begin{displaymath}
\E[\Re(\xi q_{1} \bar{\xi} q_{2})]=
4\Re(q_{1})\Re(q_{2}),
\qquad
\E[\Re(\xi q_{1} \xi q_{2})]=
-2\Re(q_{1}\overline{q_{2}}),
\end{displaymath}
\begin{displaymath}
\E[\Re(\xi q_{1}) \Re(\bar{\xi} q_{2})]=
\Re(q_{1}q_{2}),
\qquad
\E[\Re(\xi q_{1}) \Re(\xi q_{2})]=
\Re(q_{1}\overline{q_{2}}).
\end{displaymath}
\end{lemma}

\begin{lemma}
\label{LemRecursion}
Let $M$ be a random G$\beta$E$(n)$ matrix
\eqref{EqGbetaEmat}, $\beta\in\{1,2,4\}$.
Let $B, C$ be two deterministic matrices in
$\mathcal{M}_{n}(\mathbb{C})$, and 
$\widetilde{B},\widetilde{C}$ two deterministic
matrices in $\mathcal{M}_{n}(\mathbb{H})$. Then
\begin{eqnarray}
\label{EqR1}
\langle\Tr(M B M C)\rangle_{\beta=1,n}&=&
\dfrac{1}{2}\Tr(B)\Tr(C)+
\dfrac{1}{2}\Tr(B C^{\mathsf{T}}),
\\
\label{EqR2}
\langle\Tr(M B)\Tr(M C)\rangle_{\beta=1,n}&=&
\dfrac{1}{2}\Tr(B C)+
\dfrac{1}{2}\Tr(B C^{\mathsf{T}}),
\\
\label{EqR3}
\langle\Tr(M B M C)\rangle_{\beta=2,n} &=&
\Tr(B)\Tr(C),
\\
\label{EqR4}
\langle\Tr(M B)\Tr(M C)\rangle_{\beta=2,n} &=&
\Tr(B C),
\\
\label{EqR5}
\langle\Re(\Tr(M \widetilde{B} 
M \widetilde{C}))\rangle_{\beta=4,n} &=&
2\Re(\Tr(\widetilde{B}))\Re(\Tr(\widetilde{C}))-
\Re(\Tr(\widetilde{B}\widetilde{C}^{\ast})),
\\
\label{EqR6}
\langle\Re(\Tr(M \widetilde{B}))
\Re(\Tr(M \widetilde{C}))\rangle_{\beta=4,n} &=&
\dfrac{1}{2}\Re(\Tr(\widetilde{B}\widetilde{C}))+
\dfrac{1}{2}\Re(\Tr(\widetilde{B}\widetilde{C}^{\ast})).
\end{eqnarray}
\end{lemma}

\begin{proof}
If $\beta\in\{1,2\}$, the identities follow from the second moments of the entries:
\begin{displaymath}
\forall i\in \{1,\dots,n\}, 
\langle M_{ii}^{2}\rangle_{\beta=1,n} =1
=\dfrac{1}{2}+\dfrac{1}{2},
~~
\forall i\neq j\in \{1,\dots,n\}, 
\langle M_{ij}^{2}\rangle_{\beta=1,n} =
\langle M_{ij}M_{ji}\rangle_{\beta=1,n}
=\dfrac{1}{2},
\end{displaymath}
\begin{displaymath}
\forall i\in \{1,\dots,n\}, 
\langle M_{ii}^{2}\rangle_{\beta=2,n} =1,
\qquad
\forall i\neq j\in \{1,\dots,n\}, 
\langle M_{ij}M_{ji}\rangle_{\beta=2,n}
=1,
\end{displaymath}
all other second moments being zero.

If $\beta=4$, the diagonal entries are still real and commute with quaternions, so for all
$i\in \{1,\dots,n\}$,
\begin{eqnarray*}
\langle \Re(\widetilde{B}_{ii}M_{ii}
\widetilde{C}_{ii}M_{ii})\rangle_{\beta=4,n}&=&
\Re(\widetilde{B}_{ii}\widetilde{C}_{ii})
\\&=&
2\Re(\widetilde{B}_{ii})\Re(\widetilde{C}_{ii})
-
\Re\Big(\widetilde{B}_{ii}
\overline{\widetilde{C}_{ii}}\Big),
\end{eqnarray*}
\begin{eqnarray*}
\langle \Re(\widetilde{B}_{ii}M_{ii})
\Re(\widetilde{C}_{ii}M_{ii})\rangle_{\beta=4,n}&=&
\Re(\widetilde{B}_{ii})\Re(\widetilde{C}_{ii})
\\&=&
\dfrac{1}{2}\Re(\widetilde{B}_{ii}\widetilde{C}_{ii})
+
\dfrac{1}{2}\Re\Big(\widetilde{B}_{ii}
\overline{\widetilde{C}_{ii}}\Big).
\end{eqnarray*}
For the offdiagonal entries, 
$\sqrt{2}M_{ij}$ is distributed according to
\eqref{EqQuatGauss}, so we apply Lemma \ref{LemQuaternionWick}.
For all $i\neq j\in \{1,\dots,n\}$,
\begin{eqnarray*}
\langle \Re(\widetilde{B}_{ii}M_{ij}
\widetilde{C}_{jj}M_{ji})\rangle_{\beta=4,n}&=&
2\Re(\widetilde{B}_{ii})\Re(\widetilde{C}_{jj}),
\\
\langle \Re(\widetilde{B}_{ji}M_{ij}
\widetilde{C}_{ji}M_{ij})\rangle_{\beta=4,n}&=&
-\Re\Big(\widetilde{B}_{ji}
\overline{\widetilde{C}_{ji}}\Big),
\end{eqnarray*}
\begin{eqnarray*}
\langle \Re(\widetilde{B}_{ji}M_{ij})
\Re(\widetilde{C}_{ij}M_{ji})\rangle_{\beta=4,n}&=&
\dfrac{1}{2}
\Re(\widetilde{B}_{ji}\widetilde{C}_{ij}),
\\
\langle \Re(\widetilde{B}_{ji}M_{ij})
\Re(\widetilde{C}_{ji}M_{ij})\rangle_{\beta=4,n}&=&
\dfrac{1}{2}\Re\Big(\widetilde{B}_{ji}
\overline{\widetilde{C}_{ji}}\Big).
\end{eqnarray*}
All other second moments are zero.
\end{proof}

The following lemma will conclude the proof of
Theorem \ref{ThmTopoExpA}.

\begin{lemma}
\label{LemTransfo}
For $\beta\in\{1,2\}$, and $\textbf{p}$ a partition in pairs of
$\{1,\dots,\vert\nu\vert\}$,
\begin{displaymath}
\Big\langle \prod_{l=1}^{m(\nu)}
\Tr\Big(\Pi_{\nu,l,\textbf{p}}(M,A)\Big)
\Big\rangle_{\beta,n}=
\sum_{\rho\in \mathcal{R}_{\nu,\textbf{p}}}
w_{\nu,\beta}(\rho)
\prod_{\mathtt{t}\in\mathcal{T}_{\nu}(\rho)}
\Tr(\Pi_{\mathtt{t}}(A)).
\end{displaymath}
For $\beta=4$,
\begin{displaymath}
\Big\langle \prod_{l=1}^{m(\nu)}
\Re\Big(\Tr\Big(\Pi_{\nu,l,\textbf{p}}(M,\widetilde{A})
\Big)\Big)
\Big\rangle_{\beta=4,n}
=
\sum_{\rho\in \mathcal{R}_{\nu,\textbf{p}}}
w_{\nu,\beta=4}(\rho)
\prod_{\mathtt{t}\in\mathcal{T}_{\nu}(\rho)}
\Re(\Tr(\widetilde{\Pi}_{\mathtt{t}}(\widetilde{A}))).
\end{displaymath}
\end{lemma}

\begin{proof}
The proof can be done by induction on
$\vert\nu\vert/2$. If $\vert\nu\vert/2=1$, the identities above are exactly those of Lemma \ref{LemRecursion}.

If $\vert\nu\vert/2>1$, we fix 
$\textbf{p}=\{\{a_{1},b_{1}\},\dots,
\{a_{\vert\nu\vert/2},b_{\vert\nu\vert/2}\}\}$ a partition of
$\{1,\dots, \vert\nu\vert\}$ in pairs.
One can take conditional expectations with respect to 
$M^{(\vert\nu\vert/2)}$ that appears at positions
$a_{\vert\nu\vert/2}$ and $b_{\vert\nu\vert/2}$
(i.e. conditioning on 
$(M^{(1)},\dots,M^{(\vert\nu\vert/2 -1)})$). 
For that, one again applies Lemma \ref{LemRecursion}.
$\widecheck{\textbf{p}}$ will denote the partition of
$\{1,\dots, \vert\nu\vert\}\setminus
\{a_{\vert\nu\vert/2},b_{\vert\nu\vert/2}\}$ 
induced by $\textbf{p}$.
To construct a ribbon pairing $\rho\in\mathcal{R}_{\nu,\textbf{p}}$ one can proceed as follows.
\begin{enumerate}
\item First pair the ribbon half-edges
$a_{\vert\nu\vert/2}$ and 
$b_{\vert\nu\vert/2}$ in an either straight or twisted way.
\item Then contract the ribbon edge
$\{a_{\vert\nu\vert/2},b_{\vert\nu\vert/2}\}$. After this operation one gets a new family of positive integers,
$\widecheck{\nu}_{\rm str}$ if the edge 
$\{a_{\vert\nu\vert/2},b_{\vert\nu\vert/2}\}$ was straight, or
$\widecheck{\nu}_{\rm tw}$ if the edge 
$\{a_{\vert\nu\vert/2},b_{\vert\nu\vert/2}\}$ was twisted, 
and in both cases
$\vert \widecheck{\nu}_{\rm str}\vert=
\vert \widecheck{\nu}_{\rm tw}\vert=\vert\nu\vert -2$.
\item Finally take a ribbon pairing 
$\widecheck{\rho}$ in 
$\mathcal{R}_{\widecheck{\nu}_{\rm str},\widecheck{\textbf{p}}}$,
respectively
$\mathcal{R}_{\widecheck{\nu}_{\rm tw},\widecheck{\textbf{p}}}$.
\end{enumerate}
One can show that the steps (1) and (2) above correspond to taking the conditional expectation with respect to the law
$M^{(\vert\nu\vert/2)}$, and the step (3) to further taking the expectations with respect to the law of
$(M^{(1)},\dots,M^{(\vert\nu\vert/2 -1)})$. One needs to see how the weights
$w_{\nu,\beta}(\rho)$ and
$w_{\widecheck{\nu}_{\rm str},\beta}(\widecheck{\rho})$, respectively
$w_{\widecheck{\nu}_{\rm tw},\beta}(\widecheck{\rho})$, are related, with a particular attention to the case $\beta=4$. 

This edge contraction procedure appears in the proof of a quaternionic Wick formula in Bryc and Pierce \cite{BrycPierce09GSE}, Theorem 3.1. 
There the authors do in particular the verifications for the weights 
$w_{\nu,\beta=4}(\rho)$. 
Our setting is more general, and not limited to $\beta=4$, 
but the recurrence on ribbon graphs is the same. 
For completeness, we give the details in Appendix A.
\end{proof}

\begin{rem}
For $\beta\in\{1,2\}$ one can alternatively prove
Lemma \ref{LemTransfo} by expanding the product of traces,
applying Wick's rule to each scalar term and then recombining
the resulting terms.
If one wishes to proceed like that in the quaternionic case
$\beta=4$,
one needs the Wick formula \eqref{Eq quaternionic Wick trails}
for quaternionic Gaussian r.v.s; see Appendix B.
This formula is more general than the one given in 
\cite{BrycPierce09GSE}.
However, proceeding like that does not simplify the proof
of Lemma \ref{LemTransfo},
since the proof of the quaternionic 
Wick formula \eqref{Eq quaternionic Wick trails}
relies itself on an induction on ribbon edges.
Moreover, then one would need a particular treatment for
the diagonal coefficients of the matrices,
since those are real Gaussian r.v.s.
\end{rem}

\subsection{Proof of Theorem \ref{ThmTopoExpAU}}
\label{SecThm3}

We give the proof of Theorem \ref{ThmTopoExpAU}.
We use the notations of Section \ref{SubSecMatrixTwist}.
As in Section \ref{SecThm2},
for $\beta\in\{1,2,4\}$, $(M^{(i)})_{i\geq 1}$ is an i.i.d. sequence of G$\beta$E(n) matrices with distribution
\eqref{EqGbetaEmat}. 
Consider $x_{1},x_{2},\dots, x_{\vert\nu\vert}\in V$,
$\textbf{p}=\{\{a_{1},b_{1}\},\dots,
\{a_{\vert\nu\vert/2},b_{\vert\nu\vert/2}\}\}$ a partition of
$\{1,\dots, \vert\nu\vert\}$ in pairs,
with $a_{i}<b_{i}$,
and $\gamma_{1},\dots,\gamma_{\vert\nu\vert/2}$
nearest-neighbor paths on $\mathcal{G}$, with
$\gamma_{i}$ going from
$x_{a_{i}}$ to $x_{b_{i}}$.
For $l\in \{1,\dots,m(\nu)\}$,
$\Pi_{\nu,l,\textbf{p}}(M,U,A)
(\gamma_{1},\dots,\gamma_{\vert\nu\vert/2})$ 
will denote the following product of matrices.
\begin{multline*}
\Pi_{\nu,l,\textbf{p}}(M,U,A)
(\gamma_{1},\dots,\gamma_{\vert\nu\vert/2})=
\\\widehat{M}^{(\nu_{1}+\dots +\nu_{l-1}+1)}
A(\nu_{1}+\dots +\nu_{l-1}+1,
\nu_{1}+\dots +\nu_{l-1}+2)
\widehat{M}^{(\nu_{1}+\dots +\nu_{l-1}+2)}
\\
\dots A(\nu_{1}+\dots +\nu_{l}-1,
\nu_{1}+\dots +\nu_{l})
\widehat{M}^{(\nu_{1}+\dots +\nu_{l})}
A(\nu_{1}+\dots +\nu_{l},
\nu_{1}+\dots +\nu_{l-1}+1),
\end{multline*}
where the matrix $\widehat{M}^{(k)}$ is given by
\begin{displaymath}
\widehat{M}^{(k)} = 
\mathsf{hol}^{U}(\gamma_{i})M^{(i)}
\mathsf{hol}^{U}(\gamma_{i})^{\ast}
\text{ if } k=a_{i},
\qquad
\widehat{M}^{(k)} = M^{(i)}
\text{ if } k=b_{i}.
\end{displaymath}
The product 
$\Pi_{\nu,l,\textbf{p}}(M,U,\widetilde{A})
(\gamma_{1},\dots,\gamma_{\vert\nu\vert/2})$
is defined similarly, with matrices
$\widetilde{A}(k,k')$ instead of $A(k,k')$.

\begin{lemma}
\label{LemMiHol}
Then, for $\beta\in\{1,2\}$ and
$F$ a bounded measurable function 
$\mathbb{R}^{V} \rightarrow \mathbb{R}$, one has the following equality:
\begin{multline*}
\Big\langle \Big(\prod_{l=1}^{m(\nu)}
\Tr\Big(\Pi_{\nu,l}(\Phi,A)\Big)\Big)
F(\Tr(\Phi^{2})/2)\Big\rangle_{\beta,n}^{U}
\\= 
\sum_{\substack{\textbf{p} \text{ partition}
\\\text{of } \{1,\dots, \vert\nu\vert\}
\\\text{in pairs }
}}
\int\displaylimits_{\gamma_{1},\dots \gamma_{\vert\nu\vert /2}}
\Big\langle \prod_{l=1}^{m(\nu)}
\Tr\Big(
\Pi_{\nu,l,\textbf{p}}(M,U,A)
(\gamma_{1},\dots,\gamma_{\vert\nu\vert/2})
\Big)
\Big\rangle_{\beta,n}
\\
\times
\Big\langle
F\big(\Tr(\Phi^{2})/2+L(\gamma_{1})+\dots + L(\gamma_{\vert\nu\vert /2})\big)
\Big\rangle_{\beta,n}
\mu^{x_{1},x_{2},\dots, x_{\vert\nu\vert}}_{\textbf{p}}
(d\gamma_{1},\dots,d\gamma_{\vert\nu\vert /2}).
\end{multline*}
For $\beta=4$,
\begin{multline*}
\Big\langle \Big(\prod_{l=1}^{m(\nu)}
\Re\Big(\Tr\Big(\Pi_{\nu,l}(\Phi,\widetilde{A})\Big)\Big)\Big)
F(\Tr(\Phi^{2})/2)\Big\rangle_{\beta=4,n}^{U} 
\\= 
\sum_{\substack{\textbf{p} \text{ partition}
\\\text{of } \{1,\dots, \vert\nu\vert\}
\\\text{in pairs }
}}
\int\displaylimits_{\gamma_{1},\dots \gamma_{\vert\nu\vert /2}}
\Big\langle \prod_{l=1}^{m(\nu)}
\Re\Big(\Tr\Big(
\Pi_{\nu,l,\textbf{p}}(M,U,\widetilde{A})
(\gamma_{1},\dots,\gamma_{\vert\nu\vert/2})
\Big)\Big)
\Big\rangle_{\beta=4,n}
\\
\times
\Big\langle
F\big(\Tr(\Phi^{2})/2+L(\gamma_{1})+\dots + L(\gamma_{\vert\nu\vert /2})\big)
\Big\rangle_{\beta=4,n}
\mu^{x_{1},x_{2},\dots, x_{\vert\nu\vert}}_{\textbf{p}}
(d\gamma_{1},\dots,d\gamma_{\vert\nu\vert /2}).
\end{multline*}
\end{lemma}

\begin{proof}
This follows from Lemma \ref{LemIsoKLrew}. Indeed,
the action of $U(x,y)$ on
$\R^{n}$ ($\beta=1$), $\C^{n}$ ($\beta=2$),
respectively $\mathbb{H}^{n}$ ($\beta=4$), induces an action on
$E_{\beta,n}$ by conjugation:
\begin{displaymath}
M\mapsto U(x,y) M U(x,y)^{\ast}.
\end{displaymath}
$\U(x,y)$ will denote the corresponding linear orthogonal operator on $E_{\beta,n}$.
The holonomy of the connection $(\U(x,y))_{\{x,y\}\in E}$
along a path $\gamma$ is related to that of
$(U(x,y))_{\{x,y\}\in E}$ along $\gamma$ by
\begin{displaymath}
\mathsf{hol}^{\U}(\gamma)(M)=
\mathsf{hol}^{U}(\gamma) M \mathsf{hol}^{U}(\gamma)^{\ast}, 
~~ M\in E_{\beta,n}.
\end{displaymath}
So one applies Lemma \ref{LemIsoKLrew} to the connection
$(\U(x,y))_{\{x,y\}\in E}$.
\end{proof}

To finish the proof of Theorem \ref{ThmTopoExpAU}, one needs to check, for every fixed partition in pairs $\textbf{p}$, the following expressions of expectations with respect to the law of 
$(M^{(i)})_{1\leq i\leq\vert\nu\vert/2}$:
\begin{multline*}
\Big\langle \prod_{l=1}^{m(\nu)}
\Tr\Big(
\Pi_{\nu,l,\textbf{p}}(M,U,A)
(\gamma_{1},\dots,\gamma_{\vert\nu\vert/2})
\Big)
\Big\rangle_{\beta=1,n}\\=
\sum_{\rho\in \mathcal{R}_{\nu,\textbf{p}}}
w_{\nu,\beta=1}(\rho)
\prod_{\mathtt{t}\in\mathcal{T}_{\nu}(\rho)}
\Tr(
\Pi_{\mathtt{t}}(U,A)
(\gamma_{1},\dots,\gamma_{\vert\nu\vert/2})
),
\end{multline*}
\begin{multline*}
\Big\langle \prod_{l=1}^{m(\nu)}
\Tr\Big(
\Pi_{\nu,l,\textbf{p}}(M,U,A)
(\gamma_{1},\dots,\gamma_{\vert\nu\vert/2})
\Big)
\Big\rangle_{\beta=2,n}\\=
\sum_{\rho\in \mathcal{R}_{\nu,\textbf{p}}}
w_{\nu,\beta=2}(\rho)
\prod_{\scriptsize{\overrightarrow{\mathtt{t}}\in
\overrightarrow{\mathcal{T}}_{\nu}(\rho)}}
\Tr(
\Pi_{\scriptsize{\overrightarrow{\mathtt{t}}}}(U,A)
(\gamma_{1},\dots,\gamma_{\vert\nu\vert/2})
),
\end{multline*}
\begin{multline*}
\Big\langle \prod_{l=1}^{m(\nu)}
\Re\Big(\Tr\Big(
\Pi_{\nu,l,\textbf{p}}(M,U,\widetilde{A})
(\gamma_{1},\dots,\gamma_{\vert\nu\vert/2})
\Big)\Big)
\Big\rangle_{\beta=4,n}\\=
\sum_{\rho\in \mathcal{R}_{\nu,\textbf{p}}}
w_{\nu,\beta=4}(\rho)
\prod_{\mathtt{t}\in\mathcal{T}_{\nu}(\rho)}
\Re(\Tr(
\widetilde{\Pi}_{\mathtt{t}}(U,\widetilde{A})
(\gamma_{1},\dots,\gamma_{\vert\nu\vert/2})
)).
\end{multline*}
The identities above follow already from Lemma 
\ref{LemTransfo}. One has to apply the latter to matrices
$\mathcal{A}(k,k')$, respectively
$\widetilde{\mathcal{A}}(k,k')$, instead
of $A(k,k')$, respectively
$\widetilde{A}(k,k')$, where
$\mathcal{A}(k,k')$ and 
$\widetilde{\mathcal{A}}(k,k')$ are as follows:
\begin{displaymath}
\mathcal{A}(a_{i},a_{j})=
\mathsf{hol}^{U}(\gamma_{i})^{\ast}
A(a_{i},a_{j})
\mathsf{hol}^{U}(\gamma_{j}),
\qquad
\widetilde{\mathcal{A}}(a_{i},a_{j})=
\mathsf{hol}^{U}(\gamma_{i})^{\ast}\widetilde{A}(a_{i},a_{j})
\mathsf{hol}^{U}(\gamma_{j}),
\end{displaymath}
\begin{displaymath}
\mathcal{A}(a_{i},b_{j})=
\mathsf{hol}^{U}(\gamma_{i})^{\ast}A(a_{i},b_{j}),
\qquad
\widetilde{\mathcal{A}}(a_{i},b_{j})=
\mathsf{hol}^{U}(\gamma_{i})^{\ast}\widetilde{A}(a_{i},b_{j}),
\end{displaymath}
\begin{displaymath}
\mathcal{A}(b_{i},a_{j})=
A(b_{i},a_{j})
\mathsf{hol}^{U}(\gamma_{j}),
\qquad
\widetilde{\mathcal{A}}(b_{i},a_{j})=
\widetilde{A}(b_{i},a_{j})
\mathsf{hol}^{U}(\gamma_{j}),
\end{displaymath}
\begin{displaymath}
\mathcal{A}(b_{i},b_{j})=
A(b_{i},b_{j}),
\qquad
\widetilde{\mathcal{A}}(b_{i},b_{j})=
\widetilde{A}(b_{i},b_{j}).
\end{displaymath}
Then, for every $l\in\{1,\dots,m(\nu)\}$,
\begin{eqnarray*}
\Tr(\Pi_{\nu,l,\textbf{p}}(M,U,A)
(\gamma_{1},\dots,\gamma_{\vert\nu\vert/2}))&=&
\Tr(\Pi_{\nu,l,\textbf{p}}(M,\mathcal{A})))
,
\\
\Re(\Tr(\Pi_{\nu,l,\textbf{p}}(M,U,\widetilde{A})
(\gamma_{1},\dots,\gamma_{\vert\nu\vert/2})))&=&
\Re(\Tr(\Pi_{\nu,l,\textbf{p}}(M,\widetilde{\mathcal{A}}))).
\end{eqnarray*}
Further, if $\rho\in\mathcal{R}_{\nu,\textbf{p}}$ and
$\mathtt{t}\in\mathcal{T}_{\nu}(\rho)$,
\begin{eqnarray*}
\Tr(
\Pi_{\mathtt{t}}(U,A)
(\gamma_{1},\dots,\gamma_{\vert\nu\vert/2})
)&=&
\Tr(
\Pi_{\mathtt{t}}(\mathcal{A}))
),
\\
\Re(\Tr(
\widetilde{\Pi}_{\mathtt{t}}(U,\widetilde{A})
(\gamma_{1},\dots,\gamma_{\vert\nu\vert/2})
))&=&
\Re(\Tr(
\widetilde{\Pi}_{\mathtt{t}}(\widetilde{\mathcal{A}})
)).
\end{eqnarray*}
If $\rho$ is the only ribbon pairing in 
$\mathcal{R}_{\nu,\textbf{p}}$ with only straight edges
(the setting for $\beta=2$), if
$\overrightarrow{\mathtt{t}}\in
\overrightarrow{\mathcal{T}}_{\nu}(\rho)$ and
$\mathtt{t}$ is the corresponding unoriented trail, then
\begin{displaymath}
\Tr(
\Pi_{\scriptsize{\overrightarrow{\mathtt{t}}}}(U,A)
(\gamma_{1},\dots,\gamma_{\vert\nu\vert/2})
)=
\Tr(
\Pi_{\mathtt{t}}(\mathcal{A}))
).
\end{displaymath}

\section*{Appendix A: induction on ribbon edges}

Here we detail the induction on edges used in Lemma \ref{LemTransfo}.
We primarily distinguish the cases according to whether 
the ribbon half-edges $a_{\vert\nu\vert/2}$ and $b_{\vert\nu\vert/2}$
are adjacent to the same vertex (Figure \ref{FigSame})
or to two different vertices (Figure \ref{FigDiff}).
The case of one vertex corresponds to both occurrences of
$M^{(\vert\nu\vert/2)}$ being in the same trace.
This also corresponds to the left-hand side
in equations \eqref{EqR1}, \eqref{EqR3} and \eqref{EqR5}.
The case of two vertices corresponds to the two occurrences of
$M^{(\vert\nu\vert/2)}$ being in different traces.
This also corresponds to the left-hand side
in equations \eqref{EqR2}, \eqref{EqR4} and \eqref{EqR6}.

\begin{figure}[ht]
\centering
\includegraphics[scale=0.8]{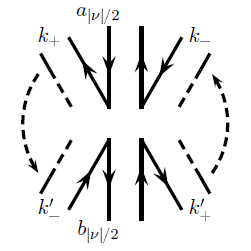}
\caption{The case when the ribbon half-edges
$a_{\vert\nu\vert/2}$ and $b_{\vert\nu\vert/2}$
are adjacent to the same vertex. 
Only the relevant vertex is represented.}
\label{FigSame}
\end{figure}

\begin{figure}[ht]
\centering
\includegraphics[scale=0.8]{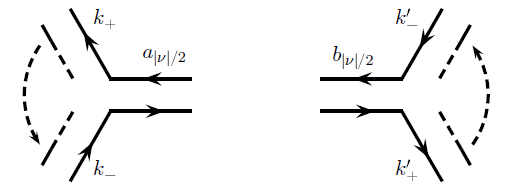}
\caption{The case when the ribbon half-edges
$a_{\vert\nu\vert/2}$ and $b_{\vert\nu\vert/2}$
are adjacent to two different vertices.
Only the two relevant vertices are represented.}
\label{FigDiff}
\end{figure}

We start with the case of one vertex.
Recall that $\rho$ denotes a ribbon pairing of
$\{1,\dots,\vert\nu\vert\}$
and $\widecheck{\rho}$ the induced ribbon pairing of
$\{1,\dots, \vert\nu\vert\}\setminus
\{a_{\vert\nu\vert/2},b_{\vert\nu\vert/2}\}$.
We further distinguish the following cases.
\begin{itemize}
\item The generic case, when the vertex has other ribbon half-edges
on both sides 
between $a_{\vert\nu\vert/2}$ and $b_{\vert\nu\vert/2}$.
Then, the straight pairing of 
$a_{\vert\nu\vert/2}$ and $b_{\vert\nu\vert/2}$ is represented on
Figure \ref{FigSameStr}.
After contracting the ribbon edge 
$\{a_{\vert\nu\vert/2},b_{\vert\nu\vert/2}\}$,
the vertex is split into two.
This corresponds to having a product of two traces in the first term on
the right-hand side 
in equations \eqref{EqR1}, \eqref{EqR3} and \eqref{EqR5}.
The orientations are preserved after contraction.
This corresponds to not having transposes or adjoints
in the first term on the right-hand side 
in equations \eqref{EqR1}, \eqref{EqR3} and \eqref{EqR5}.
Further, in case of a straight pairing, we have that
\begin{displaymath}
\vert\nu\vert
=
\vert \widecheck{\nu}_{\rm str}\vert + 2,
\qquad
m(\nu) = m(\widecheck{\nu}_{\rm str}) -1,
\qquad
f_{\nu}(\rho) =
f_{\widecheck{\nu}_{\rm str}}(\widecheck{\rho}),
\qquad
\chi_{\nu}(\rho)=
\chi_{\widecheck{\nu}_{\rm str}}(\widecheck{\rho}) -2.
\end{displaymath}
Thus,
\begin{displaymath}
w_{\nu,\beta=1}(\rho)=
\dfrac{1}{2}w_{\widecheck{\nu}_{\rm str},\beta=1}(\widecheck{\rho}),
\qquad
w_{\nu,\beta=2}(\rho)=
w_{\widecheck{\nu}_{\rm str},\beta=2}(\widecheck{\rho}),
\qquad
w_{\nu,\beta=4}(\rho)=
2 w_{\widecheck{\nu}_{\rm str},\beta=4}(\widecheck{\rho}).
\end{displaymath}
So indeed we get the coefficients in front of the first term
on the right-hand side 
in equations \eqref{EqR1}, \eqref{EqR3} and \eqref{EqR5}.

The twisted pairing of 
$a_{\vert\nu\vert/2}$ and $b_{\vert\nu\vert/2}$ is represented on
Figure \ref{FigSameTw}.
It can only occur for $\beta\in\{1,4\}$.
After contracting the ribbon edge 
$\{a_{\vert\nu\vert/2},b_{\vert\nu\vert/2}\}$,
the vertex is not divided.
This corresponds to having a single trace
in the second term on the right-hand side 
in equations \eqref{EqR1} and \eqref{EqR5}.
The orientations on one side of the contracted edge are reversed.
This corresponds to having a transpose,
respectively an adjoint,
in the second term on the right-hand side 
in equation \eqref{EqR1}, respectively \eqref{EqR5}.
Further, in case of a twisted pairing, we have that
\begin{displaymath}
\vert\nu\vert
=
\vert \widecheck{\nu}_{\rm tw}\vert + 2,
\qquad
m(\nu) = m(\widecheck{\nu}_{\rm tw}),
\qquad
f_{\nu}(\rho) =
f_{\widecheck{\nu}_{\rm tw}}(\widecheck{\rho}),
\qquad
\chi_{\nu}(\rho)=
\chi_{\widecheck{\nu}_{\rm tw}}(\widecheck{\rho}) -1,
\end{displaymath}
\begin{displaymath}
w_{\nu,\beta=1}(\rho)=
\dfrac{1}{2}w_{\widecheck{\nu}_{\rm tw},\beta=1}(\widecheck{\rho}),
\qquad
w_{\nu,\beta=4}(\rho)=
- w_{\widecheck{\nu}_{\rm tw},\beta=4}(\widecheck{\rho}).
\end{displaymath}
So we get the coefficients in front of the second term
on the right-hand side 
in equations \eqref{EqR1} and \eqref{EqR5}.
\begin{figure}[ht]
\centering
\includegraphics[scale=0.8]{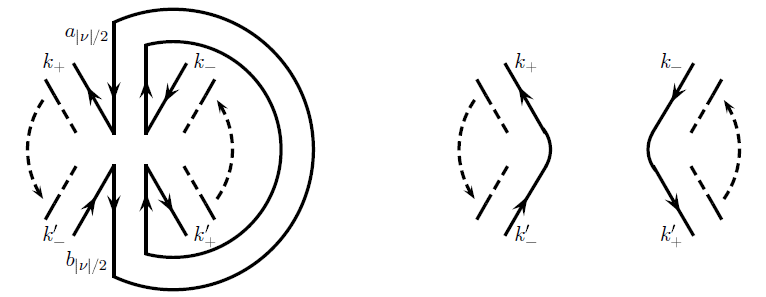}
\caption{On the left:
the ribbon half-edges
$a_{\vert\nu\vert/2}$ and $b_{\vert\nu\vert/2}$ 
belong to the same vertex and are paired in a straight way.
On the right: the result of the contraction of the corresponding straight ribbon edge. 
The vertex is split into two.
All the orientations are preserved.}
\label{FigSameStr}
\end{figure}
\begin{figure}[ht]
\centering
\includegraphics[scale=0.8]{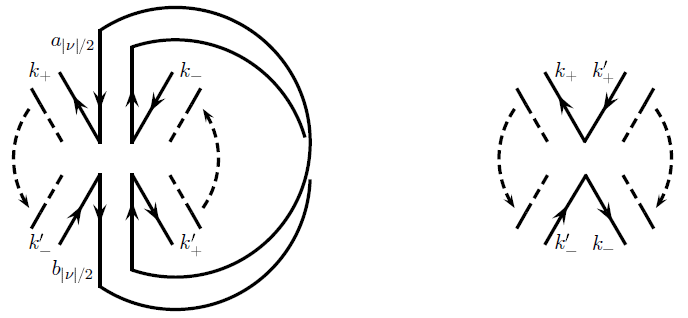}
\caption{On the left:
the ribbon half-edges
$a_{\vert\nu\vert/2}$ and $b_{\vert\nu\vert/2}$ 
belong to the same vertex and are paired in a twisted way.
On the right: the result of the contraction of the corresponding twisted ribbon edge. 
The vertex is not divided.
The orientations on one side of the contracted edge are reversed.}
\label{FigSameTw}
\end{figure}
\item The first degenerate case, 
when the vertex has other ribbon half-edges
on only one side between 
$a_{\vert\nu\vert/2}$ and $b_{\vert\nu\vert/2}$.
Then, for a straight pairing we have that
\begin{displaymath}
\vert\nu\vert
=
\vert \widecheck{\nu}_{\rm str}\vert + 2,
\qquad
m(\nu) = m(\widecheck{\nu}_{\rm str}),
\qquad
f_{\nu}(\rho) =
f_{\widecheck{\nu}_{\rm str}}(\widecheck{\rho})+1,
\qquad
\chi_{\nu}(\rho)=
\chi_{\widecheck{\nu}_{\rm str}}(\widecheck{\rho}),
\end{displaymath}
and again,
\begin{displaymath}
w_{\nu,\beta=1}(\rho)=
\dfrac{1}{2}w_{\widecheck{\nu}_{\rm str},\beta=1}(\widecheck{\rho}),
\qquad
w_{\nu,\beta=2}(\rho)=
w_{\widecheck{\nu}_{\rm str},\beta=2}(\widecheck{\rho}),
\qquad
w_{\nu,\beta=4}(\rho)=
2 w_{\widecheck{\nu}_{\rm str},\beta=4}(\widecheck{\rho}).
\end{displaymath}
Moreover, the side between 
$a_{\vert\nu\vert/2}$ and $b_{\vert\nu\vert/2}$
with no other ribbon half-edges gives rise,
after edge contraction/averaging by $M^{(\vert\nu\vert/2)}$,
to a factor consisting of a deterministic trace,
with no random matrices inside.

As for the twisted pairing, the result in this degenerate case is the same as in the generic case. Nothing changes.
\item The second degenerate case, 
when the vertex is of degree $2$ and its only ribbon half-edges
are $a_{\vert\nu\vert/2}$ and $b_{\vert\nu\vert/2}$.
Then, for a straight pairing we have that
\begin{displaymath}
\vert\nu\vert
=
\vert \widecheck{\nu}_{\rm str}\vert + 2,
\qquad
m(\nu) = m(\widecheck{\nu}_{\rm str})+1,
\qquad
f_{\nu}(\rho) =
f_{\widecheck{\nu}_{\rm str}}(\widecheck{\rho})+2,
\qquad
\chi_{\nu}(\rho)=
\chi_{\widecheck{\nu}_{\rm str}}(\widecheck{\rho})+2,
\end{displaymath}
and again,
\begin{displaymath}
w_{\nu,\beta=1}(\rho)=
\dfrac{1}{2}w_{\widecheck{\nu}_{\rm str},\beta=1}(\widecheck{\rho}),
\qquad
w_{\nu,\beta=2}(\rho)=
w_{\widecheck{\nu}_{\rm str},\beta=2}(\widecheck{\rho}),
\qquad
w_{\nu,\beta=4}(\rho)=
2 w_{\widecheck{\nu}_{\rm str},\beta=4}(\widecheck{\rho}).
\end{displaymath}
Moreover, after edge contraction/averaging by $M^{(\vert\nu\vert/2)}$,
appears a factor consisting of a product of two deterministic traces.

For a twisted pairing we have that
\begin{displaymath}
\vert\nu\vert
=
\vert \widecheck{\nu}_{\rm tw}\vert + 2,
\qquad
m(\nu) = m(\widecheck{\nu}_{\rm tw})+1,
\qquad
f_{\nu}(\rho) =
f_{\widecheck{\nu}_{\rm tw}}(\widecheck{\rho})+1,
\qquad
\chi_{\nu}(\rho)=
\chi_{\widecheck{\nu}_{\rm tw}}(\widecheck{\rho}) +1,
\end{displaymath}
and again,
\begin{displaymath}
w_{\nu,\beta=1}(\rho)=
\dfrac{1}{2}w_{\widecheck{\nu}_{\rm tw},\beta=1}(\widecheck{\rho}),
\qquad
w_{\nu,\beta=4}(\rho)=
- w_{\widecheck{\nu}_{\rm tw},\beta=4}(\widecheck{\rho}).
\end{displaymath}
Moreover, after edge contraction/averaging by $M^{(\vert\nu\vert/2)}$,
appears a factor consisting of a deterministic trace.
\end{itemize}

\medskip

Now we deal with the case of two vertices.
We further distinguish the following cases.
\begin{itemize}
\item The generic case, when at least one of the two vertices
is of degree at least two.
Then, the straight pairing of 
$a_{\vert\nu\vert/2}$ and $b_{\vert\nu\vert/2}$ is represented on
Figure \ref{FigDiffStr}.
After contracting the ribbon edge 
$\{a_{\vert\nu\vert/2},b_{\vert\nu\vert/2}\}$,
the two vertices are merged into one.
This corresponds to having a single trace in the first term on
the right-hand side 
in equations \eqref{EqR2}, \eqref{EqR4} and \eqref{EqR6}.
The orientations are preserved after contraction.
This corresponds to not having transposes or adjoints
in the first term on the right-hand side 
in equations \eqref{EqR2}, \eqref{EqR4} and \eqref{EqR6}.
Further, in case of a straight pairing, we have that
\begin{displaymath}
\vert\nu\vert
=
\vert \widecheck{\nu}_{\rm str}\vert + 2,
\qquad
m(\nu) = m(\widecheck{\nu}_{\rm str}) + 1,
\qquad
f_{\nu}(\rho) =
f_{\widecheck{\nu}_{\rm str}}(\widecheck{\rho}),
\qquad
\chi_{\nu}(\rho)=
\chi_{\widecheck{\nu}_{\rm str}}(\widecheck{\rho}).
\end{displaymath}
Thus,
\begin{displaymath}
w_{\nu,\beta=1}(\rho)=
\dfrac{1}{2}w_{\widecheck{\nu}_{\rm str},\beta=1}(\widecheck{\rho}),
\qquad
w_{\nu,\beta=2}(\rho)=
w_{\widecheck{\nu}_{\rm str},\beta=2}(\widecheck{\rho}),
\qquad
w_{\nu,\beta=4}(\rho)=
\dfrac{1}{2} w_{\widecheck{\nu}_{\rm str},\beta=4}(\widecheck{\rho}).
\end{displaymath}
So indeed we get the coefficients in front of the first term
on the right-hand side 
in equations \eqref{EqR2}, \eqref{EqR4} and \eqref{EqR6}.

The twisted pairing of 
$a_{\vert\nu\vert/2}$ and $b_{\vert\nu\vert/2}$ is represented on
Figure \ref{FigDiffTw}.
After contracting the ribbon edge 
$\{a_{\vert\nu\vert/2},b_{\vert\nu\vert/2}\}$,
the two vertices are merged into one.
This corresponds to having a single trace in the second term on
the right-hand side 
in equations \eqref{EqR2} and \eqref{EqR6}.
The orientations on one of the parent vertices are reversed after contraction.
This corresponds to the transpose, respectively adjoint,
in the second term on the right-hand side 
in equations \eqref{EqR2}, respectively \eqref{EqR6}.
Further, in case of a twisted pairing, we have that
\begin{displaymath}
\vert\nu\vert
=
\vert \widecheck{\nu}_{\rm tw}\vert + 2,
\qquad
m(\nu) = m(\widecheck{\nu}_{\rm tw}) + 1,
\qquad
f_{\nu}(\rho) =
f_{\widecheck{\nu}_{\rm tw}}(\widecheck{\rho}),
\qquad
\chi_{\nu}(\rho)=
\chi_{\widecheck{\nu}_{\rm tw}}(\widecheck{\rho}),
\end{displaymath}
\begin{displaymath}
w_{\nu,\beta=1}(\rho)=
\dfrac{1}{2}w_{\widecheck{\nu}_{\rm tw},\beta=1}(\widecheck{\rho}),
\qquad
w_{\nu,\beta=4}(\rho)=
\dfrac{1}{2} w_{\widecheck{\nu}_{\rm tw},\beta=4}(\widecheck{\rho}).
\end{displaymath}
So we get the coefficients in front of the second term
on the right-hand side 
in equations \eqref{EqR2} and \eqref{EqR6}.
\begin{figure}[ht]
\centering
\includegraphics[scale=0.8]{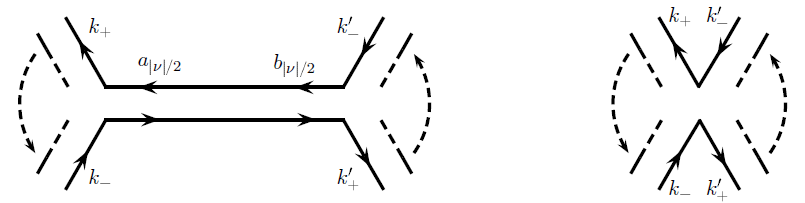}
\caption{On the left:
the ribbon half-edges
$a_{\vert\nu\vert/2}$ and $b_{\vert\nu\vert/2}$ 
belong to two different vertices and are paired in a straight way.
On the right: the result of the contraction of the corresponding straight ribbon edge. 
The two vertices are merged into one.
All the orientations are preserved.}
\label{FigDiffStr}
\end{figure}
\begin{figure}[ht]
\centering
\includegraphics[scale=0.8]{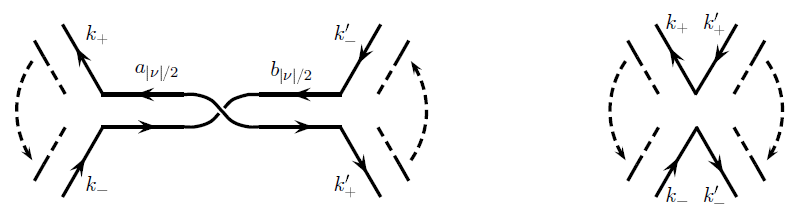}
\caption{On the left:
the ribbon half-edges
$a_{\vert\nu\vert/2}$ and $b_{\vert\nu\vert/2}$ 
belong to two different vertices and are paired in a twisted way.
On the right: the result of the contraction of the corresponding twisted ribbon edge. 
The two vertices are merged into one.
The orientations on one of the parent vertices are reversed.}
\label{FigDiffTw}
\end{figure}
\item The degenerate case, when both vertices
are of degree one.
Then, for a straight pairing we have that
\begin{displaymath}
\vert\nu\vert
=
\vert \widecheck{\nu}_{\rm str}\vert + 2,
\qquad
m(\nu) = m(\widecheck{\nu}_{\rm str})+2,
\qquad
f_{\nu}(\rho) =
f_{\widecheck{\nu}_{\rm str}}(\widecheck{\rho})+1,
\qquad
\chi_{\nu}(\rho)=
\chi_{\widecheck{\nu}_{\rm str}}(\widecheck{\rho})+2,
\end{displaymath}
and again,
\begin{displaymath}
w_{\nu,\beta=1}(\rho)=
\dfrac{1}{2}w_{\widecheck{\nu}_{\rm str},\beta=1}(\widecheck{\rho}),
\qquad
w_{\nu,\beta=2}(\rho)=
w_{\widecheck{\nu}_{\rm str},\beta=2}(\widecheck{\rho}),
\qquad
w_{\nu,\beta=4}(\rho)=
\dfrac{1}{2} w_{\widecheck{\nu}_{\rm str},\beta=4}(\widecheck{\rho}).
\end{displaymath}
The same coefficients appear for a twisted pairing.
Moreover, both in case of a straight and of a twisted pairing appears,
after edge contraction/averaging by $M^{(\vert\nu\vert/2)}$,
a factor consisting of a deterministic trace, with no random matrices inside.
\end{itemize}

\section*{Appendix B: a quaternionic Wick formula}

Here we present a Wick formula for quaternionic Gaussian r.v.s.
It is more general then the one given in
Theorem 3.1 in \cite{BrycPierce09GSE} and involves trails
(see Section \ref{SubSecRibbon}).
We are not aware whether our formula has already appeared elsewhere.

Let $(\xi_{i})_{i\geq 1}$ be a sequence of i.i.d.
quaternionic Gaussian r.v.s
distributed according to \eqref{EqQuatGauss}.
Let $\nu=(\nu_{1},\nu_{2},\dots,\nu_{m(\nu)})$, 
where for all
$l\in\{1,2,\dots,m(\nu)\}$,
$\nu_{l}\in\mathbb{N}\setminus\{0\}$, and $\vert\nu\vert$ is even.
Let $(\eta_{k})_{1\leq k\leq \vert\nu\vert}$
be a family of r.v.s,
where each $\eta_{k}$ is either a r.v.
$\xi_{i}$, or a r.v. $\bar{\xi_{i}}$,
for some $i\in\{1,\dots ,\vert\nu\vert\}$.
We also consider $\vert\nu\vert$ deterministic quaternions
\begin{displaymath}
q(1,2),\dots,q(\nu_{1}-1,\nu_{1}),
q(\nu_{1},1),
\end{displaymath}
\begin{displaymath}
q(\nu_{1}+1,\nu_{1}+2),\dots,
q(\nu_{1}+\nu_{2}-1,\nu_{1}+\nu_{2}),
q(\nu_{1}+\nu_{2},\nu_{1}+1),
\end{displaymath}
\begin{displaymath}
\dots,
q(\vert\nu\vert-\nu_{m(\nu)}+1,
\vert\nu\vert-\nu_{m(\nu)}+2),
\dots,
q(\vert\nu\vert-1,\vert\nu\vert),
q(\vert\nu\vert,\vert\nu\vert-\nu_{m(\nu)}+1).
\end{displaymath}
Note that for each
$l\in \{1,\dots,m(\nu)\}$ such that
$\nu_{l}=1$, we have a single quaternion
$q(\nu_{1}+\dots +\nu_{l},\nu_{1}+\dots +\nu_{l})$.
For $l\in \{1,\dots,m(\nu)\}$,
$\Pi_{\nu,l}(\eta,q)$ will denote the product
\begin{multline*}
\Pi_{\nu,l}(\eta,q)=
\eta_{\nu_{1}+\dots +\nu_{l-1}+1}
q(\nu_{1}+\dots +\nu_{l-1}+1,
\nu_{1}+\dots +\nu_{l-1}+2)
\eta_{\nu_{1}+\dots +\nu_{l-1}+2}
\\
\dots q(\nu_{1}+\dots +\nu_{l}-1,
\nu_{1}+\dots +\nu_{l})
\eta_{\nu_{1}+\dots +\nu_{l}}
q(\nu_{1}+\dots +\nu_{l},
\nu_{1}+\dots +\nu_{l-1}+1).
\end{multline*}
In case $\nu_{l}=1$,
$\Pi_{\nu,l}(\eta,q)=\eta_{\nu_{1}+\dots +\nu_{l}}
q(\nu_{1}+\dots +\nu_{l},\nu_{1}+\dots +\nu_{l})$.
We will express the moment
\begin{equation}
\label{Eq Mom quater}
\E\Big[
\prod_{l=1}^{m(\nu)}
\Re\Big(\Pi_{\nu,l}(\eta,q)\Big)
\Big].
\end{equation}
Theorem 3.1 in \cite{BrycPierce09GSE} gives an expression in case
all the quaternions $q(k,k')$ equal $1$.
Note that in order for the moment \eqref{Eq Mom quater}
to be different from $0$,
the number of occurrences of $\xi_{i}$
plus the number of occurrences of $\bar{\xi_{i}}$
in the family $(\eta_{k})_{1\leq k\leq \vert\nu\vert}$
has to be even for each $i$.
$\mathcal{R}_{\nu,\eta}$ will denote the subset of
$\mathcal{R}_{\nu}$
made of all ribbon pairings $\rho$
such that for each straight ribbon edge
with labels $k,k'$,
$\eta_{k} = \overline{\eta_{k'}}$ a.s.,
and for each twisted ribbon edge with labels $k,k'$,
$\eta_{k} = \eta_{k'}$ a.s.
Given a ribbon pairing
$\rho\in \mathcal{R}_{\nu,\eta}$
and a trail $\mathtt{t}\in\mathcal{T}_{\nu}(\rho)$,
let $\widetilde{\Pi}_{\mathtt{t}}(q)$
denote a product of quaternions of the form
$q(k,k')$ or
$\overline{q(k',k)}$.
For each sequence
$k\rightarrow k'$ in the trail $\mathtt{t}$ we add the factor
$q(k,k')$ to the product, and for each
sequence $k\leftarrow k'$, we add the factor
$\overline{q(k',k)}$, 
all by respecting cyclic order of the trail.

\begin{propB}
With the notations above, we have that
\begin{equation}
\label{Eq quaternionic Wick trails}
\E\Big[
\prod_{l=1}^{m(\nu)}
\Re\Big(\Pi_{\nu,l}(\eta,q)\Big)
\Big]
=
2^{\vert\nu\vert/2}
\sum_{\rho\in \mathcal{R}_{\nu,\eta}}
w_{\nu,\beta=4}(\rho)
\prod_{\mathtt{t}\in\mathcal{T}_{\nu}(\rho)}
\Re\Big(\widetilde{\Pi}_{\mathtt{t}}(q)\Big)
.
\end{equation}
\end{propB}

\begin{proof}
On can proceed by induction on ribbon edges as in the proof
of Lemma \ref{LemTransfo}.
The case $\vert\nu\vert=2$ is given by
Lemma \ref{LemQuaternionWick}.
\end{proof}

\section*{Acknowledgements}

This work was supported by the French National Research Agency (ANR) grant
within the project MALIN (ANR-16-CE93-0003).

The author thanks the two anonymous reviewers for their helpful remarks on the previous version of this paper.

\bibliographystyle{alpha}
\bibliography{titusbibnew}

\end{document}